\documentclass[a4paper, reqno]{amsart}
\usepackage[foot]{amsaddr}
\usepackage{fullpage}
\usepackage{graphicx}
\usepackage{color}
\usepackage{amsthm,amssymb}

\usepackage{tikz}
\usepackage{tikz-cd}
\usepackage[hidelinks]{hyperref}
\usepackage{mathtools}
\usepackage{xcolor}
\usepackage{algpseudocode}
\usepackage{algorithm}
\usepackage{enumerate}
\usepackage{nicematrix}

\usepackage[colorinlistoftodos,prependcaption,textsize=tiny]{todonotes}

\newcommand {\mm}[1]   {\ifmmode{#1}\else{\mbox{\(#1\)}}\fi}

\newcommand {\scalprod}[2] {{\langle #1 , #2 \rangle}}

\newcommand{\Rspace}        {\mm{{\mathbb R}}}

\newcommand{\Zspace}        {\mm{{\mathbb Z}}}
\newcommand{\Bgroup}[2]     {\mm{\sf B}_{#1}{({#2})}}
\newcommand{\Cgroup}[2]     {\mm{\sf C}_{#1}{({#2})}}
\newcommand{\Hgroup}[2]     {\mm{\sf H}_{#1}{({#2})}}
\newcommand{\Zgroup}[2]     {\mm{\sf Z}_{#1}{({#2})}}

\newcommand{\fgroup}[2]     {\mm{\sf f}_{#1}^{#2}}
\newcommand{\Bottleneck}[2] {\mm{{W_\infty}{({#1},{#2})}}}

\newcommand{\domain}[2]     {\mm{{\rm dom}{({#1},{#2})}}}
\newcommand{\Voronoi}[2]    {\mm{{\rm Vor}_{#1}{({#2})}}}

\newcommand{\Delaunay}[2]   {\mm{{\rm Del}_{#1}{({#2})}}}
\newcommand{\Alpha}[2]      {\mm{{\rm Alf}_{#1}{({#2})}}}
\newcommand{\AlphaSub}[3]   {\mm{{\rm Alf}_{#1}{({#2},{#3})}}}
\newcommand{\Nerve}[1]      {\mm{{\rm Nerve}{({#1})}}}
\newcommand{\Dgm}[2]        {\mm{{\rm Dgm}_{#1}{({#2})}}}
\newcommand{\Sd}[1]         {\mm{{\rm Sd}\,{{#1}}}}
\newcommand{\Star}[2]       {\mm{{\rm st}{({#1},{#2})}}}

\newcommand{\Lift}[2]       {\mm{{#1}_{#2}^{\,\rotatebox{90}{\scriptsize $\triangleright$}}}}
\newcommand{\Plank}[3]      {\mm{{\rm Plank}_{#1}{({#2},{#3})}}}
\newcommand{\CPlank}[4]     {\mm{{\rm Plank}_{#1}{({#2},{#3},{#4})}}}
\newcommand{\VCPlank}[4]    {\mm{{\rm Plank}_{#1}^{V}{({#2},{#3},{#4})}}}
\newcommand{\Ball}[2]       {\mm{{B}_{#1}{({#2})}}}
\newcommand{\VBall}[3]      {\mm{{B}_{#1}^{V}{({#2},{#3})}}}
\newcommand{\Radiusf}       {\mm{{\rm Rad}}}

\newcommand{\Kfun}          {\mm{{f_K}}}
\newcommand{\Lfun}          {\mm{{f_L}}}
\newcommand{\Mfun}          {\mm{{f_M}}}
\newcommand{\KLfun}         {\mm{{f_{K,L}}}}
\newcommand{\KMfun}         {\mm{{f_{K,M}}}}
\newcommand{\LMfun}         {\mm{{f_{L,M}}}}

\newcommand{\rank}[1]       {\mm{{\rm rank\,}{#1}}}

\newcommand{\card}[1]       {\mm{{\#}{#1}}}
\newcommand{\dime}[1]       {\mm{\rm dim\,}{#1}}

\newcommand{\Span}[1]       {\mm{\rm span\,}{#1}}

\newcommand{\image}[2]      {\mm{\rm im}_{#1}{{\,}#2}}
\newcommand{\kernel}[2]     {\mm{\rm ker}_{#1}{\,#2}}
\newcommand{\coker}[2]      {\mm{\rm cok}_{#1}{\,#2}}
\newcommand{\kkk}           {\mm{\sf k}}

\newcommand{\Edist}[2]      {\mm{\|{#1}-{#2}\|}}
\newcommand{\norm}[1]       {\mm{\|{#1}\|}}

\newcommand{\problemP}[1]{(P_{#1})}
\newcommand{\problemD}[1]{(D_{#1})}

\newcommand{\Skip}[1]       {}

\definecolor{blue-red}{rgb}{0.8, 0.00, 0.95}

\theoremstyle{definition}

\newtheorem{theorem}{Theorem}
\numberwithin{theorem}{section}

\newtheorem{proposition}[theorem]{Proposition}

\newtheorem{lemma}[theorem]{Lemma}
\newtheorem{corollary}[theorem]{Corollary}

\newtheorem{definition}[theorem]{Definition}

\numberwithin{equation}{section}

\title{Chromatic Alpha Complexes}

\author{Sebastiano Cultrera di Montesano$^1$}
\email{$^1$sebastiano.cultrera@ist.ac.at}
\author{Ond\v{r}ej Draganov$^2$}
\email{$^2$ondrej.draganov@ist.ac.at}
\author{Herbert Edelsbrunner$^3$}
\email{$^3$herbert.edelsbrunner@ist.ac.at}
\author{Morteza Saghafian$^4$}
\email{$^4$morteza.saghafian@ist.ac.at}

\address{$^{1,2,3,4}$ISTA (Institute of Science and Technology Austria), Kloster\-neu\-burg, Austria}


\begin{document}

\begin{abstract}
    Motivated by applications in the sciences, we study finite chromatic sets in Euclidean space from a topological perspective. 
    Based on the persistent homology for images, kernels and cokernels, we design provably stable homological quantifiers that describe the geometric micro- and macro-structure of how the color classes mingle. 
    These can be efficiently computed using chromatic variants of Delaunay and alpha complexes, and code that does these computations is provided.
\end{abstract}

\maketitle

\section{Introduction}
\label{sec:1}

This paper takes a topological approach to quantifying spatial interactions between several point sets, which we distinguish by color.
The aim is the development of a mathematical language to answer questions like:
``how, how often, and at what scale do blue points surround groups of red points?'', or ``are there cycles made out of blue, red, and green points that make essential use of all three colors?''. We tackle these questions from a multi-scale homological perspective, with the goal of disentangling patterns such as the ones shown in Figure~\ref{fig:patterns-filled}.
\begin{figure}[htb]
  \centering \vspace{0.1in}
  \resizebox{!}{1.8in}{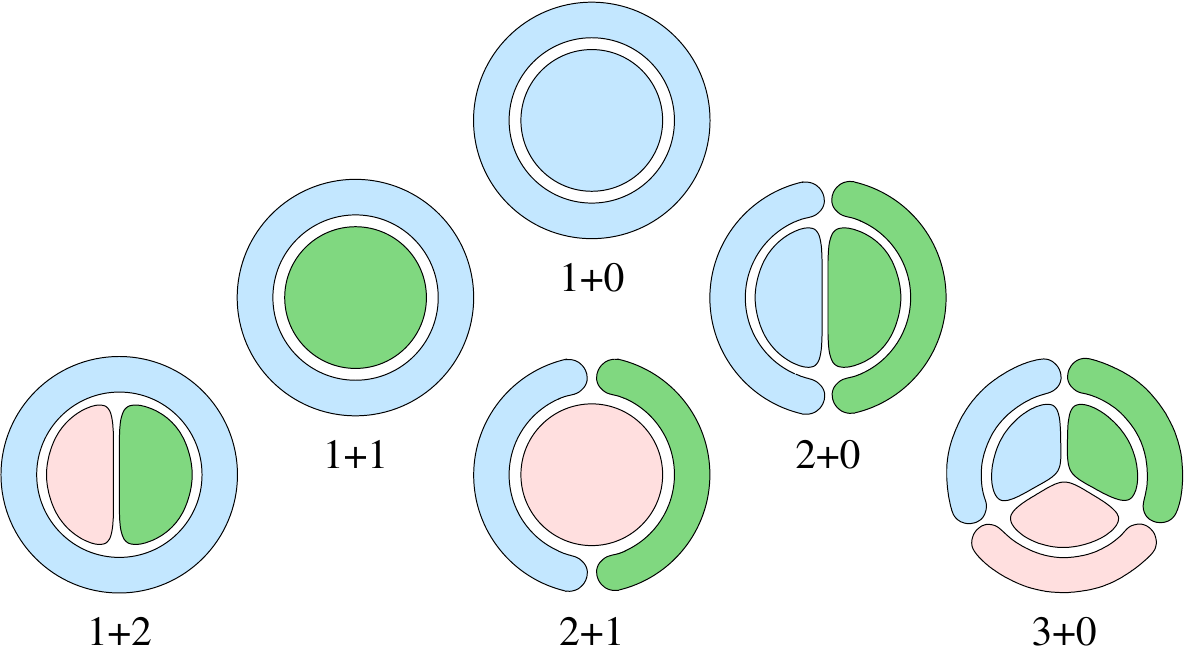}
  \vspace{-0.05in}
  \caption{\footnotesize{Mingling patterns distinguished by the number of colors forming a cycle and the number of additional colors filling this cycle.
  The drawings are caricatures of similar patterns for cycles different from circles and fillings different from disks.
  The patterns are but a first attempt to differentiate types of interactions, and they are by no means precise or exhaustive.
  For example, two additional colors can fill a cycle in at least two different ways (see the pattern of type 1+2): in a collaboration as suggested in the drawing, or each individually, like two different patterns of type 1+1.}}
  \label{fig:patterns-filled}
\end{figure}

\smallskip
One of the motivations for this work is the recent growth of interest in \emph{spatial biology}, which combines the biological properties of cells with their locations.
An example is the \emph{tumor immune microenvironment} \cite{Bin18} in cancer research, which focuses on the interplay between tumor and immune cells.
Can we identify as well as quantify patterns in the interaction between cell types that correlate with clinical outcomes?
Another biological process that raises similar mathematical questions is \emph{cell sorting} (the natural segregation of cells by types), which in early development is studied for instance in \cite{Mai12} and in the context of somitogenesis is mentioned in \cite{Mia23}.
This motivates the study of \emph{chromatic point sets}, in which the points represent cells and colors represent their types.

\smallskip
Our specific approach is based on the formulation of persistent homology in terms of growing balls~\cite{EdHa10}, which is often used in topological data analysis to describe and quantify spatial arrangements of (mono-chromatic) point sets. 
The idea behind the construction is to transform a discrete set of points into a nested sequence of topological spaces.
This is achieved by growing balls centered at the input points (from zero to infinite radius), which yields a sequence of progressively larger shapes.
Such a nested sequence of spaces is called a \emph{filtration}, and it becomes a sequence of vector spaces connected by linear maps when we apply homology with field coefficients.
The latter is also known as a \emph{persistence module}.
Studying the induced maps rather than just individual vector spaces, we can not only identify radii at which topological features appear and disappear, but also pair these events to quantify for how long each feature persists in the filtration.

\smallskip
When the data is bi-chromatic, say \emph{red} and \emph{blue}, we have two sets of growing balls at our disposal. 
A natural way to relate them is to consider the inclusion map between the union of balls of one color, say the blue ones, into the union of balls of both colors.
Like in the mono-chromatic case, we apply homology with field coefficients and get two persistence modules together with connecting maps, which are induced by inclusions relating the two filtrations.
The connecting maps carry important information about the mingling of the two point sets. 
For example, if a cycle is present in the blue filtration at a certain radius, it may or may not also be present in the red and blue filtration. 
If it is, such a cycle will be in the image of the connecting map, while if it is not, it will be in the kernel. 
More generally, given a pair of filtrations related by inclusions, we can look at the persistent homology of the subspace, the full space, the relative space, as well as the kernels, the images, and the cokernels of the connecting maps \cite{CEHM09}. 
We call the resulting collection the \emph{6-pack of persistence diagrams}, which we use to capture different aspects of the mingling between geometric sets.
One contribution of this paper is the study of relations between the six persistence diagrams composing the 6-pack, such as linear relations between their \emph{$1$-norms}, which are the sums of persistences of the points in the diagrams (Theorem \ref{thm:norm_relations}).

\smallskip
Just like alpha complexes are a possible discrete model for the union of balls in the mono-chromatic setting \cite{EdKiSe83, EdMu94}, we seek a chromatic variant that enables the computation of the 6-pack of persistent diagrams. 
At first sight, this seems problematic as the red alpha complex does not include into the alpha complex of the union of red and blue points. 
Similarly, taking the full red subcomplex does not work either, as it does not capture the homotopy type of the union of red balls.
We circumvent these limitations by using a third type of complex,
the \emph{chromatic Delaunay mosaic}, which was introduced for two colors by Reani and Bobrowski \cite{ReBo21} and which we extended beyond two colors in \cite{BCDES22}. 
This mosaic uses an extra dimension for each color beyond the first to capture the interaction between colors.
Counter-balancing the increase in dimension, \cite{BCDES22} shows that the complexity of the mosaic is moderate for a small number of colors.
For example, this paper gives linear bounds on the expected size for points in two dimensions randomly colored by a constant number of colors.
Building on the results in \cite{BCDES22}, we show that the chromatic Delaunay mosaic can be equipped with a radius function whose sublevel sets capture the alpha complexes of different color classes as well as their interactions.
Within this setting, we show that the radius function on the chromatic Delaunay mosaic can be computed in linear time assuming the dimension and the number of colors is constant (Theorem \ref{thm:chromatic_radius_function_in_linear_time}), and that it has the combinatorial structure of a generalized discrete Morse function (Theorem~\ref{thm:chromatic_radius_functions_are_generalized_discrete_Morse}).
Code that implements these algorithms is available at \cite{DrMa23}---see the github repository for details.

\smallskip
The entire development could have been based on chromatic variants of the \v{C}ech complex, with almost no differences, except that the complex has size exponential in the number of vertices, making computational experiments of the kind presented in this paper infeasible.
Similarly, we could have used chromatic variants of the Vietoris--Rips complex, but the complexes would again be significantly larger, and we would have to cope with topological artifacts, which at this time are not understood. 
More recently, Dowker complexes and witness complexes have been suggested as possible candidates for encoding spatial relations in the tumor microenvironment \cite{Sto23}. 
The main limitation of these complexes is their lack of stability: perturbing a point set can produce a very different filtration.
Moreover, the Dowker complex is limited to the study of two interacting point sets, while the witness complex requires a choice of ``landmark points" and it is not clear how to choose those in practice.

\smallskip \noindent \textbf{Outline.}
Section~\ref{sec:2} reviews the alpha complex in the mono-chromatic case.
Section~\ref{sec:3} extends this construction to the chromatic case. 
Section~\ref{sec:4} proves that the radius function on the chromatic Delaunay mosaic is generalized discrete Morse.
Section~\ref{sec:5} studies the persistent homology of the chromatic alpha complexes, with an emphasis on the two and three colors settings.
Section~\ref{sec:6} concludes the paper.

\section{Mono-chromatic Point Sets}
\label{sec:2}

In this section, we recall several standard definitions and relevant results used in topological data analysis of point sets with no extra color labels.

\subsection{Voronoi Tessellation and Delaunay Complex}
\label{sec:2.1}

Letting $A \subseteq \Rspace^d$ be a finite set of points, the \emph{Voronoi domain} of $a \in A$, denoted $\domain{a}{A}$, is the set of points $x \in \Rspace^d$ that satisfy $\Edist{x}{a} \leq \Edist{x}{b}$ for all $b \in A$.
Observe that $\domain{a}{A}$ is the intersection of finitely many closed half-spaces and therefore a closed convex polyhedron.
The \emph{Voronoi tessellation} of~$A$, denoted $\Voronoi{}{A}$, is the collection of Voronoi domains defined by the points in $A$.
These domains cover $\Rspace^d$ while their interiors are pairwise disjoint.
Nevertheless, a collection of these polyhedra may overlap in a shared face, which we refer to as a \emph{Voronoi cell}.
For a generic set, $A$, the dimension of a Voronoi cell is determined by the number of Voronoi domains that share it.

\begin{definition}[Conventional Genericity]
  \label{dfn:conventional_genericity}
  We call a point set, $A \subseteq \Rspace^d$, \emph{generic} if every $p$-sphere, with $0 \leq p < d$, passes through at most $p+2$ points of $A$.
\end{definition}
\noindent Then, indeed, the common intersection of any $p+1$ Voronoi domains is either empty or a convex polyhedron of dimension $d-p$.
Note that our notion of genericity allows for more than $p+1$ points on a $p$-dimensional affine subspace.

\smallskip
The \emph{Delaunay complex}, denoted $\Delaunay{}{A}$, is the simplicial complex with vertex set $A$ that contains a simplex corresponding to each collection of Voronoi domains with non-empty common intersection.
It is isomorphic to the nerve of the Voronoi domains,
\begin{align}
  \Nerve{\Voronoi{}{A}} &= \{ \nu \subseteq \Voronoi{}{A} \mid \bigcap \nu \neq \emptyset \}.
\end{align}
The Delaunay complex can also be characterized with empty spheres passing through points.
A ($(d-1)$-dimensional) sphere, $S$, is \emph{empty} if all points of $A$ lie on or outside the sphere, so there are no points inside the sphere; that is, no points in the interior of the ball bounded by the sphere.
Furthermore, we say the sphere \emph{passes through} the points that lie on the sphere; see the left panel of Figure~\ref{fig:different_radii} for an empty $1$-sphere that passes through three points.

\begin{lemma}
  \label{lem:empty_sphere_characterization}
  Let $A \subseteq \Rspace^d$ be a finite set of points, and $\nu \subseteq A$. 
  Then $\nu$ is a simplex in the Delaunay complex iff there exists an empty sphere that passes through all points in $\nu$.
\end{lemma}
\begin{proof}
  Let $\nu \in \Delaunay{}{A}$.
  By definition of Delaunay complex, $\bigcap_{a \in \nu} \domain{a}{A} \neq \emptyset$, and we let $x$ be a point in this common intersection of Voronoi domains. 
  Then $x$ has the same distance to all points in $\nu$ and the same or a larger distance to all other points in $A$.
  Hence, $x$ is the center of an empty sphere that passes through all points in $\nu$, and possibly also through other points of $A$.
  Each of the above implications can be reversed, which implies that $x$ is in the intersection of Voronoi domains of points in $\nu$ iff $x$ is the center of an empty ball that passes through $\nu$.
\end{proof}

\noindent \emph{Remark.}
Notice that the Delaunay complex generally differs from the \emph{dual} of the Voronoi tessellation, which contains a $p$-dimensional cell for each $(d-p)$-dimensional intersection of Voronoi domains.
We call this dual the \emph{Delaunay mosaic}: it contains the convex hull of a subset of points as a cell whenever the corresponding collection of Voronoi domains is maximal with this necessarily non-empty common intersection.
However, for a generic set of points, the Delaunay mosaic is the Delaunay complex of the points.
Throughout this paper, we will work with the Delaunay complex rather than the mosaic, and we will appeal to genericity in cases this is necessary.
Delaunay mosaic can also be characterized with empty spheres passing through points: the convex hull of $\nu$ is a cell in the Delaunay mosaic iff there is an empty sphere that passes through the points in $\nu$ and through no other points of $A$.

\subsection{Alpha Complex}
\label{sec:2.2}

There can be more than one empty sphere passing through the vertices of a simplex, $\nu\in\Delaunay{}{A}$, but there is a unique \emph{smallest empty sphere} that passes through the points in $\nu$.
This yields a \emph{radius function} on the Delaunay complex, $\Radiusf \colon \Delaunay{}{A}\to \Rspace$, which maps each simplex to the radius of the smallest empty sphere that passes through its vertices.
The \emph{alpha complex}, $\Alpha{r}{A}\subseteq\Delaunay{}{A}$, is the sublevel set consisting of all simplices with radius at most $r$. Note that for $r\leq R$ we have $\Alpha{r}{A} \subseteq \Alpha{R}{A}$.

\smallskip
Let $\Ball{r}{a}$ be the $d$-ball with radius $r$ centered at $a \in \Rspace^d$. 
The \emph{Voronoi ball} of $a \in A$ with radius $r$ is this ball restricted to the Voronoi domain: $\VBall{r}{a}{A} = \Ball{r}{a} \cap \domain{a}{A}$. 
Using the correspondence in Lemma~\ref{lem:empty_sphere_characterization}, it is straightforward to observe that the alpha complex, $\Alpha{r}{A}$, is isomorphic to the nerve of the Voronoi balls of $A$ with radius $r$. 
The Nerve Theorem \cite{Bor48} then implies that
\begin{align}
    \Alpha{r}{A} &\simeq \bigcup\nolimits_{a\in A} \VBall{r}{a}{A} = \bigcup\nolimits_{a\in A} \Ball{r}{a}.
\end{align}
Furthermore, a generalization of this theorem guarantees that the homotopy equivalences for different radii, $r \leq R$, commute with the inclusions on both sides.
From the perspective of topology, we can equivalently study the union of growing balls or its discrete counterpart, the growing alpha complex.

\begin{figure}[htb]
    \centering
    \includegraphics[width=.4\textwidth]{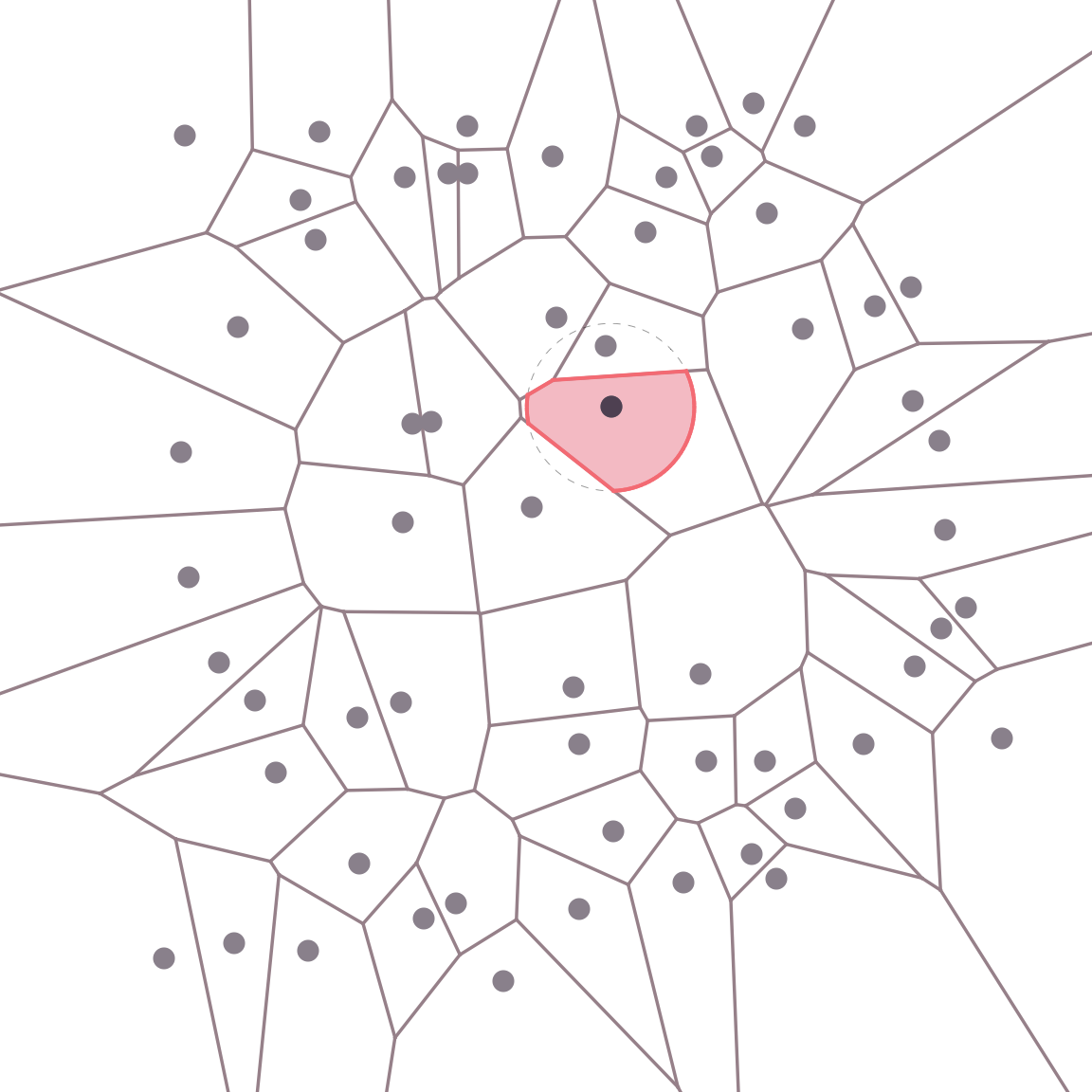}
    \hspace{5mm}
    \includegraphics[width=.4\textwidth]{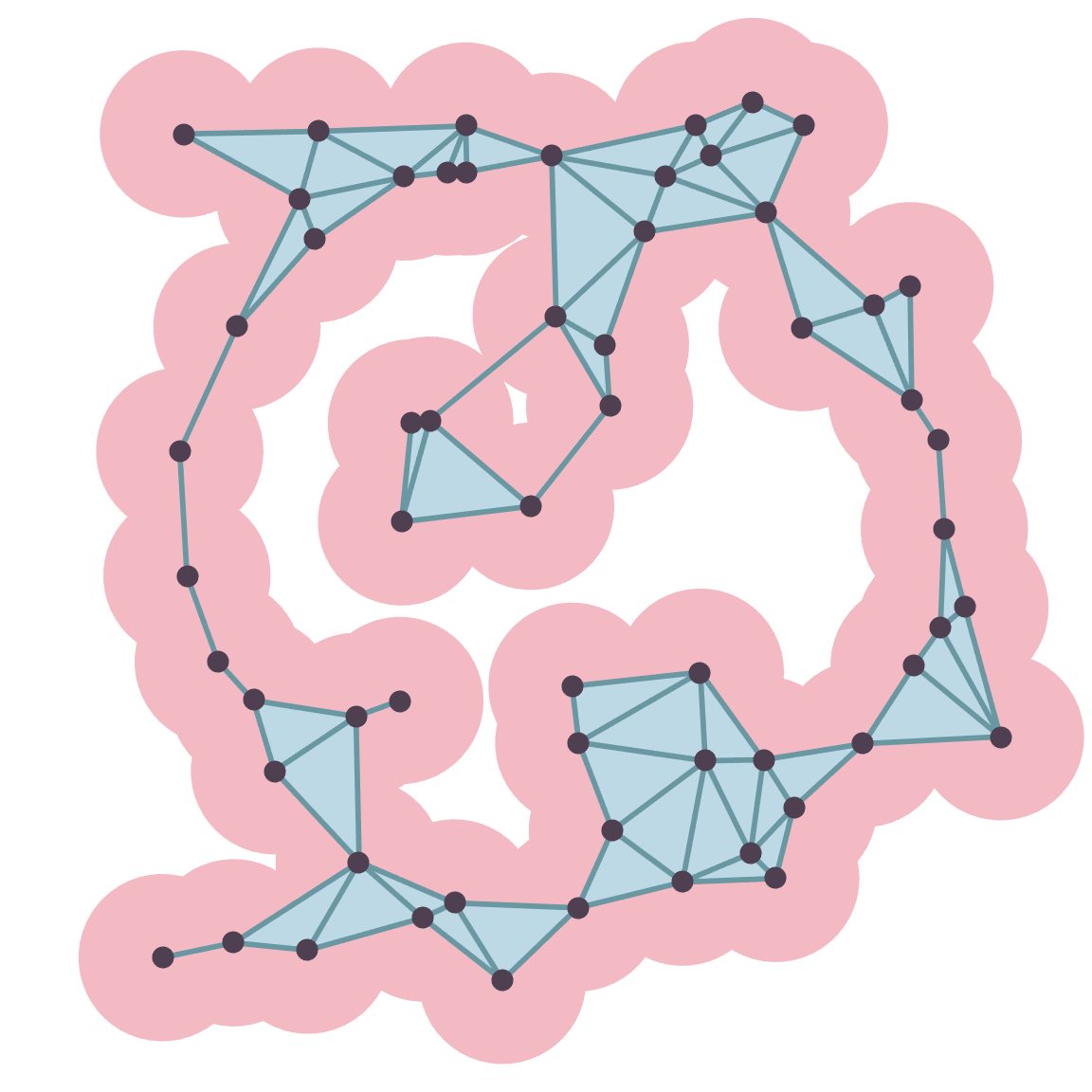}
    \caption{\footnotesize{On the \textit{left}: a set, $A \subseteq \Rspace^2$, together with its Voronoi tessellation, $\Voronoi{}{A}$, and one Voronoi ball highlighted. 
    On the \textit{right}: the union of disks, $\bigcup_{a\in A} \Ball{r}{a}$, with the alpha complex, $\Alpha{r}{A}$, superimposed.}}
    \label{fig:voronoi_ball_alpha_complex}
\end{figure}

\smallskip
For a finite point set, the Delaunay complex is finite, and hence the filtration of different alpha complexes is finite. 
To better understand the structure of this filtration, it is helpful to understand the level sets of the radius function. 
Given simplices $\alpha \subseteq \gamma$ in a simplicial complex $K$, write $[\alpha, \gamma]$ for the simplices $\beta$ that satisfy $\alpha \subseteq \beta \subseteq \gamma$; that is: $[\alpha, \gamma]$ is an \emph{interval} in the face poset of $K$.
Given a monotonic function $f \colon K \to \Rspace$, an \emph{interval of $f$} is such an interval on which $f$ is constant, and it is \emph{maximal} if it is not strictly contained in a larger interval of $f$.
\begin{definition}\label{dfn:generalised_discrete_morse_function}
  A monotonic function on a simplicial complex, $f \colon K \to \Rspace$, is \emph{generalized discrete Morse} if the maximal intervals of $f$ partition $K$.
\end{definition}
Equivalently, $f$ is generalized discrete Morse if every level set, $K_t = f^{-1} (t)$, is a disjoint union of maximal intervals.
Discrete Morse theory was introduced by Forman~\cite{For98} and later generalized by Freij~\cite{Fre09} before it became a research area on its own; see the book by Scoville~\cite{Sco19}.

\begin{proposition}[{\cite[Corollary 4.6]{BaEd17}}]
  \label{prop:radius_functions_are_generalized_discrete_Morse}
    Let $A \subseteq \Rspace^d$ be finite and generic\footnote{
      The general position condition in \cite{BaEd17} is more restrictive than Definition~\ref{dfn:conventional_genericity}, but for this claim the latter suffices.
      Indeed, Proposition~\ref{prop:radius_functions_are_generalized_discrete_Morse} is a special case of Theorem~\ref{thm:chromatic_radius_functions_are_generalized_discrete_Morse}, and the chromatic genericity in Definition~\ref{dfn:chromatic_genericity} can be relaxed to the genericity in Definition~\ref{dfn:conventional_genericity} if there is only one color ($k = 0$).}.
    Then $\Radiusf \colon \Delaunay{}{A} \to \Rspace$ is a generalized discrete Morse function.
\end{proposition}

A benefit of this result is a better understanding of how the homology changes throughout the filtration.
Consider constructing the alpha complex by adding one interval at a time in the order of their radius values. 
The homology of the complex does not change when the interval contains two or more simplices, and it necessarily changes when the interval contains only one simplex.

\section{Chromatic Point Sets}
\label{sec:3}

The main concept in this section is the
chromatic alpha complex, which generalizes the bi-chromatic construction in \cite{ReBo21} to three and more colors.
A crucial ingredient is the radius function on the chromatic Delaunay complex, whose sublevel sets are the chromatic alpha complexes.
The section follows the logical structure of Section~\ref{sec:2} to clearly showcase the analogies between the mono-chromatic and the more general, chromatic settings. 
Indeed, we will see that the chromatic definitions match the definitions in the previous section when we only consider one color.

\smallskip
A \emph{chromatic point set} is a mapping $\chi \colon A \to \sigma$, in which $A\subseteq \Rspace^d$ is a finite point set, and $\sigma$ is a set of colors.
We usually write $s = \card{\sigma} - 1$ and $\sigma = \{ 0, 1, \ldots, s \}$. 
Furthermore, we write $A_j = \chi^{-1}(j)$ for the subset of points of a given color $j \in \sigma$. 
We fix this notation throughout the section.

\subsection{Chromatic Voronoi Tessellation and Chromatic Delaunay Complex}
\label{sec:3.1}

The \emph{chromatic Voronoi tessellation}, 
$\Voronoi{}{\chi}$, is the collection of Voronoi domains, $\domain{a}{A_{\chi(a)}}$, for all points $a \in A$. 
In other words, $\Voronoi{}{\chi}$ is the union of the $\Voronoi{}{A_j}$, over all $j \in \sigma$.
Differently colored Voronoi domains can have overlapping interiors. 
Indeed, every point in $\Rspace^d$ is covered by at least $s+1$ different domains from $\Voronoi{}{\chi}$, namely at least one domain for each color.

\smallskip
The \emph{chromatic Delaunay complex}, denoted $\Delaunay{}{\chi}$, contains a simplex $\nu \subseteq A$ if the common intersection of the corresponding domains is non-empty.
It is isomorphic to the nerve of the chromatic Voronoi tessellation.
A direct analogy of the characterization of the Delaunay complex with empty spheres is the characterization of the chromatic Delaunay complex with what we call empty stacks. 
A \emph{$\sigma$-stack} in $\Rspace^d$ is a collection of $s+1$ concentric $(d-1)$-spheres, one for each color in $\sigma$; see Figure~\ref{fig:stack}. 
We drop $\sigma$ from the notation if it is clear from the context. 
The \emph{radius} of the stack is the maximum radius of its spheres, and its \emph{center} is the common center of the spheres. 
We label the spheres $S_j$, $j \in \sigma$, and say the stack is \emph{empty} if $S_j$ is empty of points in $A_j = \chi^{-1} (j)$, for each $j \in \sigma$.
We say the stack \emph{passes through} $\nu \subseteq A$ if $S_j$ passes through $\nu \cap A_j$, for each $j \in \sigma$.
\begin{lemma}\label{lem:empty_stack_characterization}
    Let $\chi \colon A \to \sigma$ be a chromatic point set in $\Rspace^d$, write $A_j = \chi^{-1}(j)$, and let $\nu \subseteq A$ be a collection of points. 
    Then $\nu \in \Delaunay{}{\chi}$ iff there exists an empty stack of spheres that passes through $\nu$. 
\end{lemma}
\begin{proof}
    Let $\nu_j = \nu \cap A_j$ be the $j$-colored points in $\nu$, for each $j \in \sigma$. 
    By Lemma~\ref{lem:empty_sphere_characterization}, the existence of an empty sphere, $S_j$, with center $x$ that passes through $\nu_j$ is equivalent to $x$ being in the intersection of the corresponding Voronoi domains: $x \in \bigcap_{a \in \nu_j} \domain{a}{A_j}$. 
    Therefore, there exists an empty stack passing through $\nu$ centered at $x$ iff $x \in \bigcap_{a \in \nu_j} \domain{a}{A_j}$ for each $j \in \sigma$. 
    This is the defining property of $\nu$ being in $\Delaunay{}{\chi}$, namely that $\bigcap_{a \in \nu} \domain{a}{A_{\chi(a)}}$ is non-empty.
\end{proof}

\begin{figure}[htb]
    \centering
    \includegraphics[width=.4\textwidth]{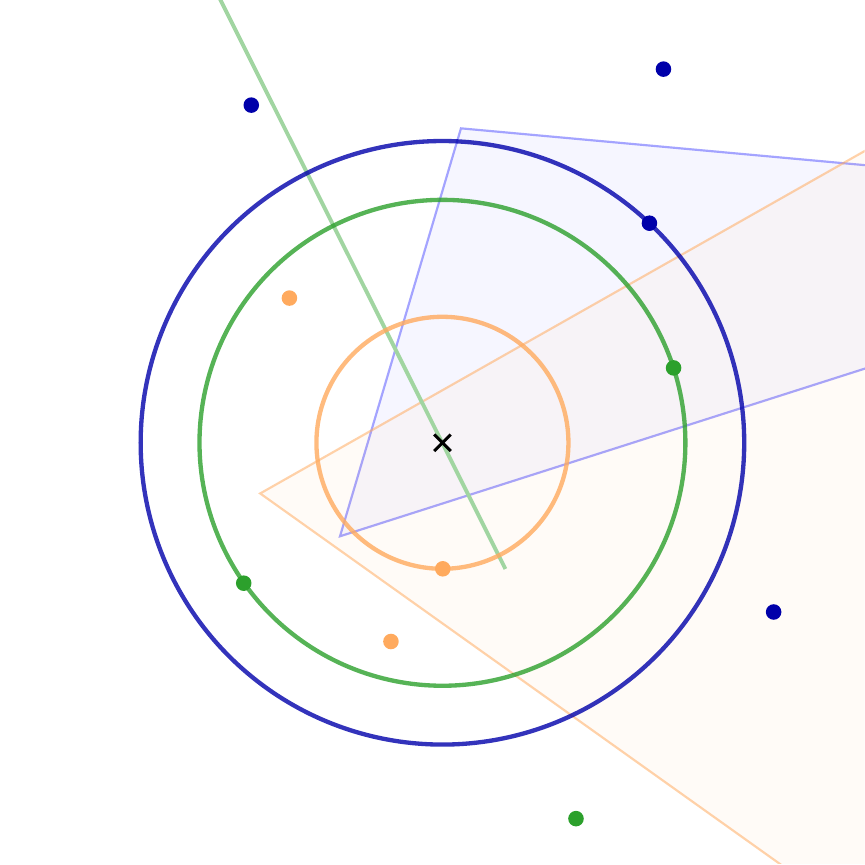} \hspace{10mm}
    \includegraphics[width=.4\textwidth]{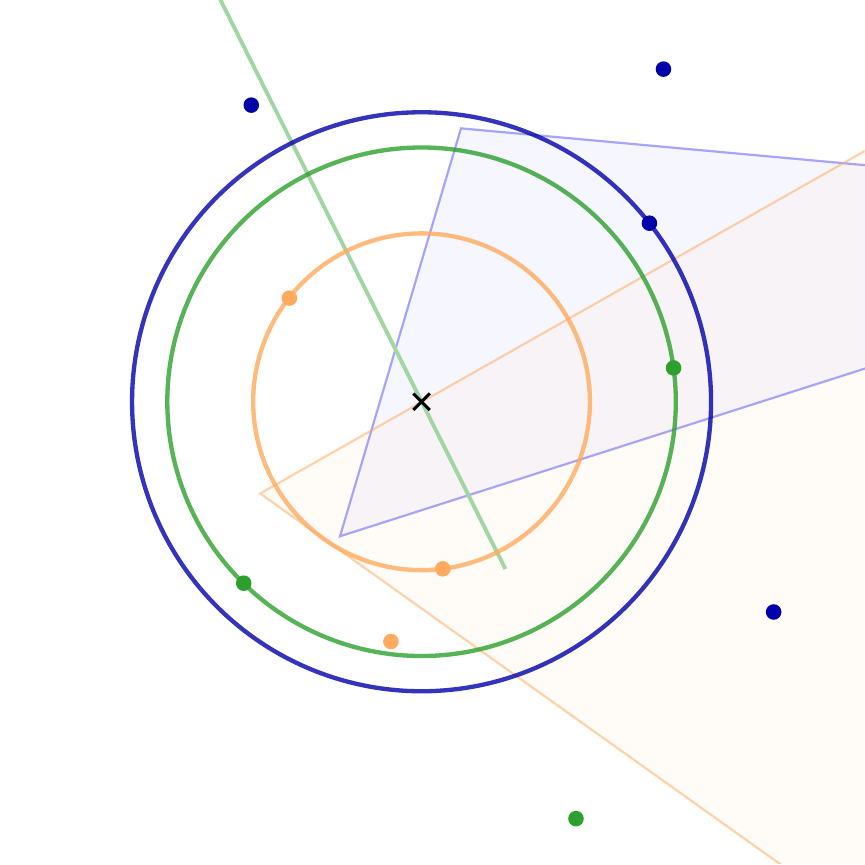}
    \caption{
    \footnotesize{Two empty stacks in $\Rspace^2$ that pass through one blue point, two green points, and one orange point forming a simplex $\nu\in\Delaunay{}{\chi}$. 
    (In fact, the stack on the \emph{right} passes through \emph{two} orange points, and it also passes through the \emph{one} orange point that the \emph{left} orange circle passes through.)
    The set of centers of all empty stacks that pass through these four points is the intersection of three Voronoi cells: a blue $2$-cell, a green $1$-cell, and an orange $2$-cell.
    The \emph{right panel} shows the smallest empty stack in this collection: its center lies on the boundary of the intersection of Voronoi cells, which is the reason why one of its circles passes through an extra point.
    }}
    \label{fig:stack}
\end{figure}

\subsection{Chromatic Alpha Complex}
\label{sec:3.2}

Like in the mono-chromatic setting, we define chromatic alpha complexes as sublevel sets of the radius function defined on the chromatic Delaunay complex.
We recall that the radius of a stack is the radius of its largest sphere.
\begin{definition}\label{dfn:chromatic_alpha_complex}
    Let $\chi \colon A \to \sigma$ be a chromatic point set, and $\Radiusf \colon \Delaunay{}{\chi} \to \Rspace$ the \emph{radius function} defined by mapping $\nu \in \Delaunay{}{\chi}$ to the radius of the smallest empty stack that passes through $\nu$.
    The \emph{chromatic alpha complex} of $\chi$ with radius $r \in \Rspace$ is $\Alpha{r}{\chi} = \Radiusf^{-1} [0,r]$.
\end{definition}
\noindent Using the empty spheres and empty stacks characterizations (Lemmas~\ref{lem:empty_sphere_characterization} and \ref{lem:empty_stack_characterization}), we can see a clear relation between alpha complexes and chromatic alpha complexes. 
If there exists an empty $(d-1)$-sphere, $S$, of radius $r$ passing through $\nu\subseteq A$, then there also exists an empty stack of radius $r$ passing through $\nu$. 
Indeed, we can take $S_j = S$ for each $j \in \sigma$. 
Similarly, an empty sphere, $S$, that passes through points $\nu \subseteq A_j$ is itself an empty stack when we set $S_i$ to be a sphere with zero radius for $i \neq j$.
However, the same simplex can have a different radius in $\Delaunay{}{A}$ and in $\Delaunay{}{\chi}$: the smallest empty sphere can have strictly larger radius than the smallest empty stack passing through the same points; see Figure~\ref{fig:different_radii}.
We formulate the above observation in a slightly more general form.
\begin{lemma}\label{lem:alpha_union_in_alpha_chromatic}
    Let $\chi \colon A \to \sigma$ be a chromatic point set in $\Rspace^d$, and $\tau\subseteq\sigma$ a subset of the colors.
    \begin{enumerate}[(i)]
        \item \label{lem:alpha_union_in_alpha_chromatic_i} Let $\eta \colon \sigma \to \tau$ be a merging of colors. 
        Then $\Alpha{r}{\eta \circ \chi} \subseteq \Alpha{r}{\chi}$ for all $r$.
        \item \label{lem:alpha_union_in_alpha_chromatic_ii} Let $\chi|\tau \colon \chi^{-1}(\tau) \to \tau$ be the restriction of $\chi$ to colors in $\tau$. 
        Then $\Alpha{r}{\chi|\tau} \subseteq \Alpha{r}{\chi}$ for all $r$.
    \end{enumerate}
    In particular, $\Alpha{r}{A} \subseteq \Alpha{r}{\chi}$, and $\Alpha{r}{A_j} \subseteq \Alpha{r}{\chi}$ for every $j \in \sigma$, in which $A_j = \chi^{-1}(j)$.
\end{lemma}
\begin{proof}
    To see (i), let $(S_i)_{i \in \tau}$ be an empty stack of radius $r$ that passes through the points in $\nu \subseteq A$;
    it witnesses $\nu\in\Alpha{r}{\eta\circ\chi}$.
    Then $(S_{\eta(j)})_{j \in \sigma}$ is an empty stack of radius $r$ that passes though the points in $\nu$; it witnesses $\nu\in\Alpha{r}{\chi}$.
    To see (ii), let $(S_i)_{i \in \tau}$ be an empty stack of radius $r$ that witnesses $\nu\in\Alpha{r}{\chi|\tau}$.
    Adding zero-radius spheres for the colors $j \in \sigma \setminus \tau$, we get an empty stack that witnesses $\nu \in \Alpha{r}{\chi}$.
\end{proof}

\begin{figure}[htb]
    \centering
    \includegraphics[width=.32\textwidth]{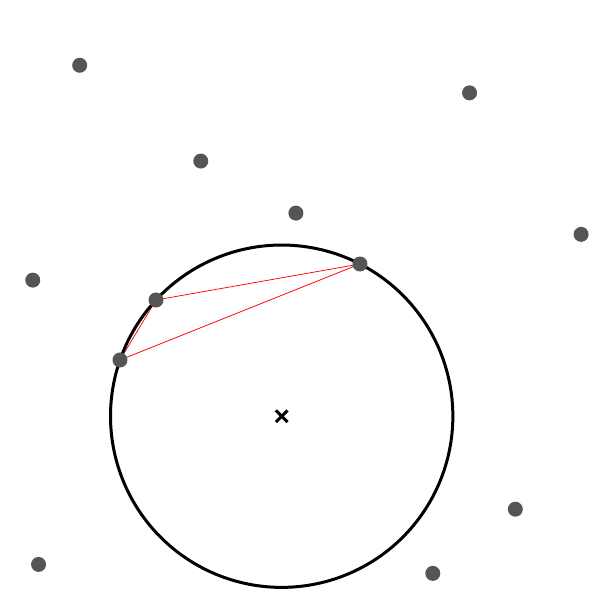} \hspace{22mm}
    \includegraphics[width=.32\textwidth]{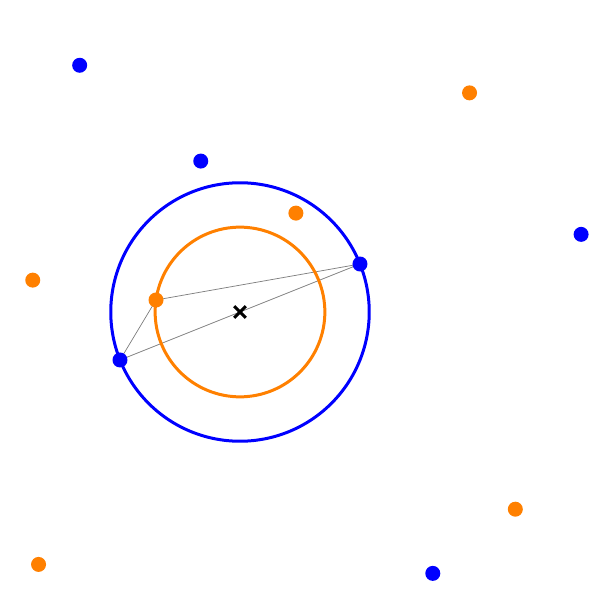}
    \caption{\footnotesize{An obtuse triangle with two blue points and an orange point at the obtuse angle.
    The smallest empty sphere that passes through the three points on the \emph{left} has strictly larger radius than the smallest empty stack that passes through the three points on the \emph{right.}
    Therefore, the triangle belongs to both, the Delaunay complex and the chromatic Delaunay complex, but it has a different value under the two radius functions.}}
    \label{fig:different_radii}
\end{figure}

An important reason why the alpha complex is useful in the mono-chromatic setting is its correspondence to the union of balls growing from the input points. 
From the topological point of view, studying the growing union of balls is equivalent to studying the growing alpha complex.
In the following, we draw an analogous connection for chromatic alpha complexes. 
One important distinction is that there is more structure to be preserved in the chromatic setting: not only the topological spaces themselves, but also how they are related to each other. 
For example, for a bi-chromatic point set as in Figure~\ref{fig:disks_chromatic}, we study the inclusion of the union of the blue disks into the union of all disks. 
We prove that we can equivalently study the inclusions of the blue alpha complex into the chromatic alpha complex.

\begin{figure}[htb]
    \centering
    \includegraphics[width=.4\textwidth]{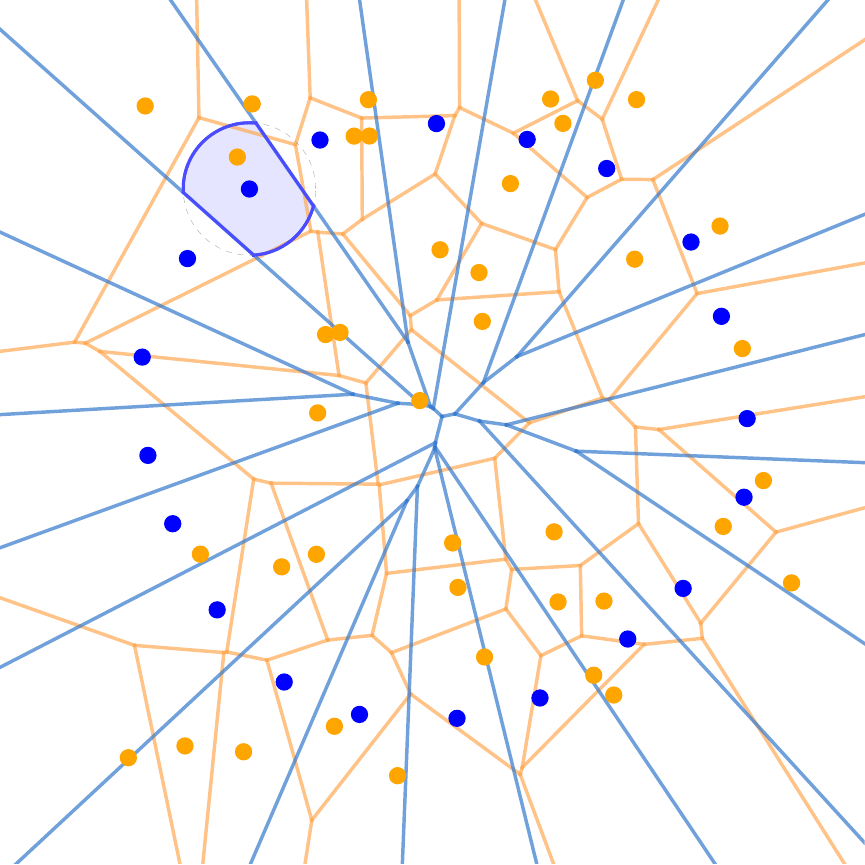}
    \hspace{5mm}
    \includegraphics[width=.4\textwidth]{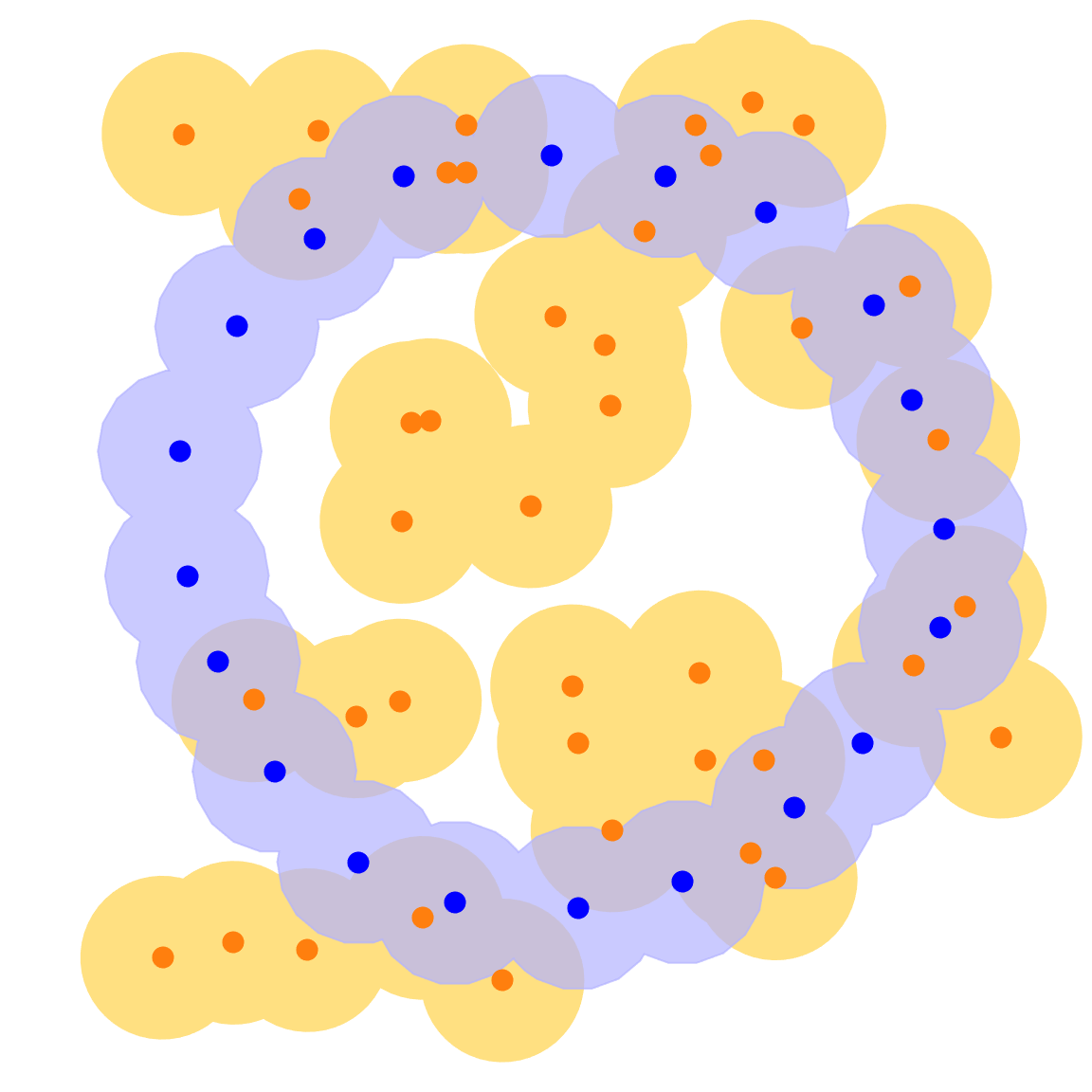}
    \caption{\footnotesize{On the \emph{left}: a chromatic set together with the Voronoi tessellations of the blue and orange points overlaid, and one chromatic Voronoi ball highlighted. 
    On the \emph{right}: the union of blue and the union of orange disks.
    }}
    \label{fig:disks_chromatic}
\end{figure}

\smallskip
For a point $a \in A$ in a chromatic set $\chi \colon A \to \sigma$, we define its \emph{(chromatic) Voronoi ball} of radius $r$ as the intersection of the ball of radius $r$ with the Voronoi domain within its color class:
\begin{align}
  \VBall{r}{a}{\chi} &= \VBall{r}{a}{A_{\chi(a)}} = \Ball{r}{a} \cap \domain{a}{A_{\chi(a)}}.
\end{align}
Let $\nu \subseteq A$ be a set of points. 
Like in the proof of Lemma~\ref{lem:empty_stack_characterization}, we observe that $x$ is the center of an empty stack of radius $r$ passing through $\nu$ iff $x$ is contained in the intersection of the Voronoi balls of radius $r$ centered at the points in $\nu$; that is: $x \in \bigcap_{a \in \nu} \VBall{r}{a}{\chi}$.
This implies that $\Alpha{r}{\chi}$ is isomorphic to the nerve of all Voronoi balls $\VBall{r}{a}{\chi}$, $a\in A$. 
Since the union of the Voronoi balls is the same as the union of the balls, the Nerve Theorem yields the following:

\begin{lemma}\label{lem:homotopy_equivalence}
    $\Alpha{r}{\chi} \simeq \Alpha{r}{A}$, and both are homotopy equivalent to the union of balls, $\bigcup_{a \in A}\Ball{r}{a}$.
\end{lemma}

Unlike the alpha complex, $\Alpha{r}{A}$, the chromatic alpha complex, $\Alpha{r}{\chi}$, contains $\Alpha{r}{A_j}$ as a subcomplex for each color $j \in \sigma$. 
We claim that this reflects the inclusion of the union of $j$-colored balls into the union of all balls, which allows us to study that inclusion on the discrete side. 
Since the complexes involved are defined as nerves, we can use a version of the Nerve Theorem (e.g.\ Theorem B in  \cite{BKRR23}) to show that the inclusions commute with the homotopy equivalences.
\begin{lemma}\label{lem:homotopy_equivalence_commutativity}
    For every color $j \in \sigma$ and radius $r$, the following diagram commutes:\[
    \begin{tikzcd}
        \Alpha{r}{\chi} \arrow[r, leftrightarrow, "\simeq"] &
        \bigcup\nolimits_{a\in A} \Ball{r}{a} \\
        \Alpha{r}{A_j} \arrow[u, hook] \arrow[r, leftrightarrow, "\simeq"] &
        \bigcup\nolimits_{a\in A_j} \Ball{r}{a} \arrow[u, hook]
    \end{tikzcd}\]
\end{lemma}
See Theorem~\ref{thm:homotopy_equivalence_commutativity_general} in Section~\ref{sec:3.4} for a generalization of this statement and a proof.

\subsection{Lifting Construction}
\label{sec:3.3}

Next we recall the lifting construction in \cite{BCDES22,ReBo21}, which sheds light on the structure of the chromatic Delaunay complex and can be used for its computation.
\begin{figure}[htb]
  \centering
  \includegraphics[width=.35\textwidth]{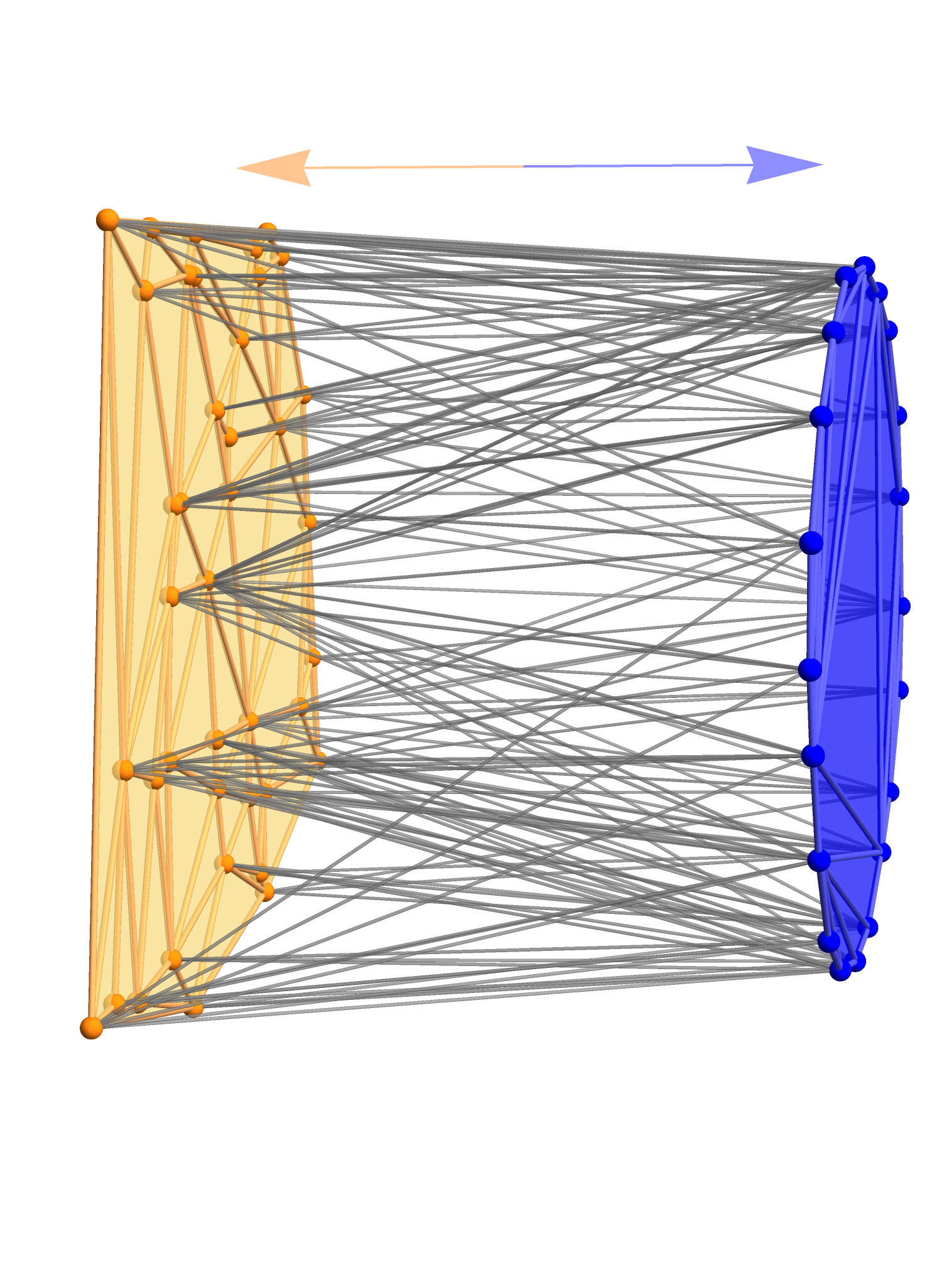}
  \hspace{10mm}
  \includegraphics[width=.35\textwidth]{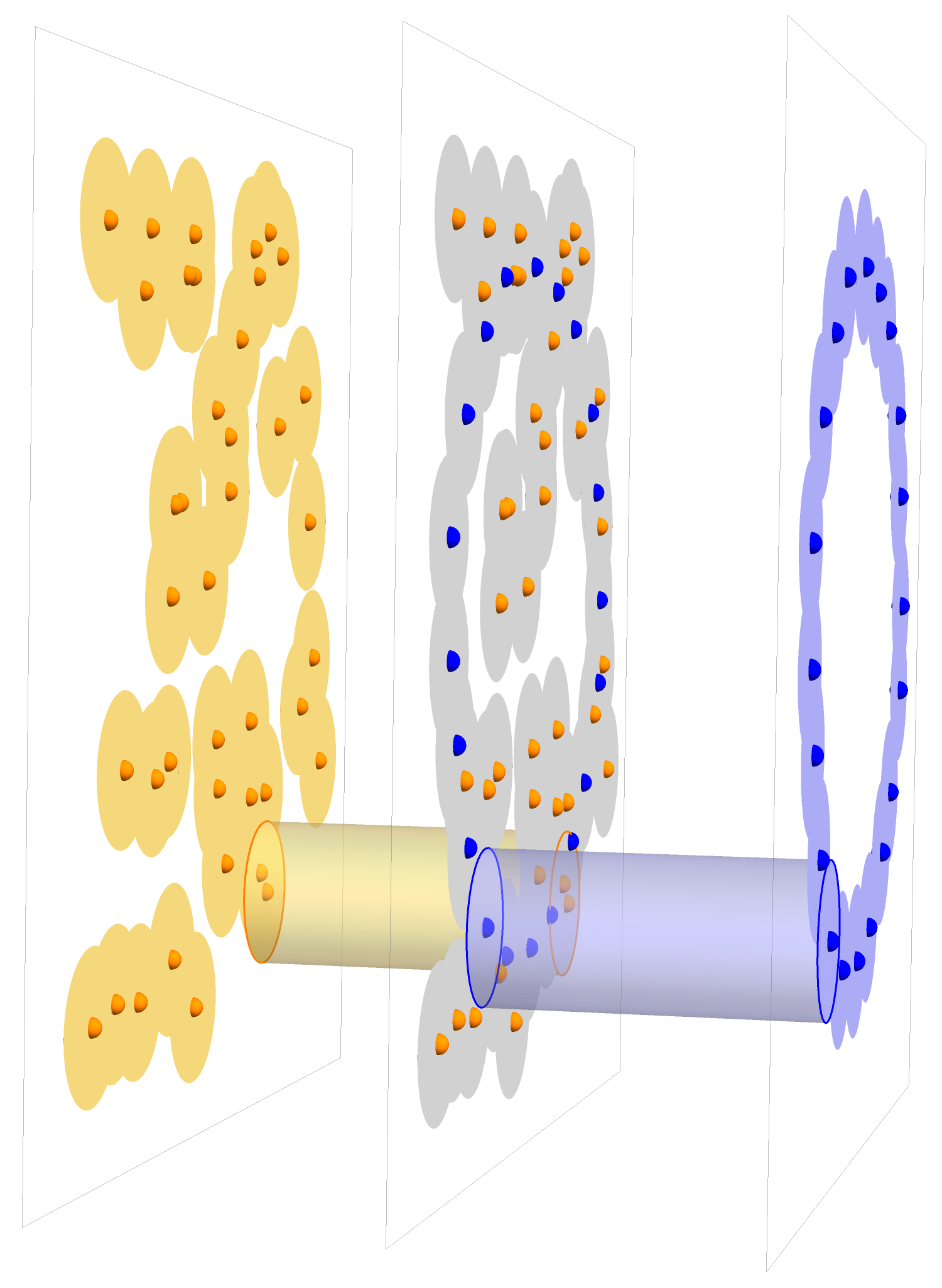}
  \vspace{-4mm}
  \caption{\footnotesize{The chromatic lifting for points in $\Rspace^2$ with two colors.
    \emph{Left:} the chromatic Delaunay complex embedded in three dimensions using the lifted points as vertices; it is isomorphic to the Delaunay complex of the lifted points. 
    The triangles and tetrahedra with more than one color are left unfilled for clarity.
    \emph{Right:} the union of disks of radius $r$, for each color separately on the two sides, and for both colors together in the middle.
    As defined shortly in Section~\ref{sec:3.4}, each plank (solid cylinder) connects a disk in the middle with the same disk on one of the sides.
    We show one such cylinder for each color.}}
  \label{fig:lifitng_2d_2_colors}
\end{figure}

\smallskip
Let $\chi \colon A \to \sigma$ be a chromatic set of points, with $A \subseteq \Rspace^d$ finite and colors $\sigma = \{0,1,\ldots,s\}$.
To separate the colors, we use points $u_0, u_1, \ldots, u_s$ in $\Rspace^s$.
For convenience, we assume these points are the vertices of the standard $s$-simplex embedded in $\Rspace^s$, but the construction would work for any affinely independent collection of $s+1$ points.
Let $\Rspace^d$ and $\Rspace^s$ be spanned by the first $d$ and last $s$ coordinate vectors of $\Rspace^{d+s}$, respectively; that is: we treat $\Rspace^d$ and $\Rspace^s$ as orthogonal subspaces of $\Rspace^{d+s}$.
Write $A_j = \chi^{-1} (j)$, and set $\Lift{A}{j} = A_j + u_j$ for $0 \leq j \leq s$.
Then $\Lift{A}{} = \Lift{A}{0} \cup \Lift{A}{1} \cup \ldots \cup \Lift{A}{s}$ is a finite set in $\Rspace^{d+s}$, and we call it the \emph{chromatic lifting of $\chi$}.
We claim that the chromatic Delaunay complex, $\Delaunay{}{\chi}$, is the standard Delaunay complex, $\Delaunay{}{\Lift{A}{}}$, after identifying the lifted vertices with their original counterparts.
This gives us a straightforward way to compute $\Delaunay{}{\chi}$---we lift the points and use a standard algorithm to compute the Delaunay complex---and an intuitive view on the structure of $\Delaunay{}{\chi}$; see Figures~\ref{fig:lifitng_2d_2_colors} and \ref{fig:lifitng_1d_3_colors}. 
The claimed equality is easy to prove when we use the characterizations via empty stacks and empty spheres.
\begin{lemma}\label{lem:stacks_and_lifted_spheres}
  Let $\chi \colon A \to \sigma$ be a chromatic point set in $\Rspace^d$ and $\Lift{A}{} \subseteq \Rspace^{d+s}$ be its chromatic lifting. 
  There exists an empty stack of $s+1$ $(d-1)$-spheres that pass through the points in $\nu \subseteq A$ iff there exists an empty $(d+s-1)$-sphere that passes through the lifted points in $\Lift{\nu}{} = \left\{ a+u_{\chi(a)} \mid a\in\nu \right\}$.
\end{lemma}
\begin{proof}
  There is a 1-to-1 correspondence between the stacks of concentric $(d-1)$-spheres in $\Rspace^d$ and the $(s+d-1)$-spheres in $\Rspace^{d+s}$ that have a non-empty intersection with $\Rspace^d + u_j$ for each $j \in \sigma$. 
  Indeed, if $S$ is such an $(s+d-1)$-sphere, then we get the stack by setting $S_j$ to be the intersection of $S$ with $\Rspace^d+u_j$ projected back to $\Rspace^d$, for each $j \in \sigma$, and if $(S_j)_{j \in \sigma}$ is a stack, its spheres share a common center, so we can find a sphere $S$ whose intersection with $\Rspace^d+u_j$ is $S_j + u_j$, for each $j \in \sigma$.

  This correspondence implies that $S$ is empty of points in $\Lift{A}{}$ and passes through the points of $\Lift{\nu}{}$ iff $(S_j)_{j \in \sigma}$ is empty of points in $\chi$ and passes through the points in $\nu$.
\end{proof}

Lemmas~\ref{lem:empty_sphere_characterization}, \ref{lem:empty_stack_characterization}, and \ref{lem:stacks_and_lifted_spheres} imply the following result.
\begin{corollary}\label{cor:chromatic_delaunay_is_delaunay}
    Let $\chi \colon A \to \sigma$ be a chromatic point set and $\Lift{A}{}$ its chromatic lifting.
    Then $\Delaunay{}{\chi}$ and $\Delaunay{}{\Lift{A}{}}$ are isomorphic, with the isomorphism defined by mapping $a \in A$ to $\Lift{a}{} = a+u_{\chi(a)}$.
\end{corollary}

\begin{figure}[htb]
    \centering
    \includegraphics[width=.22\textwidth]{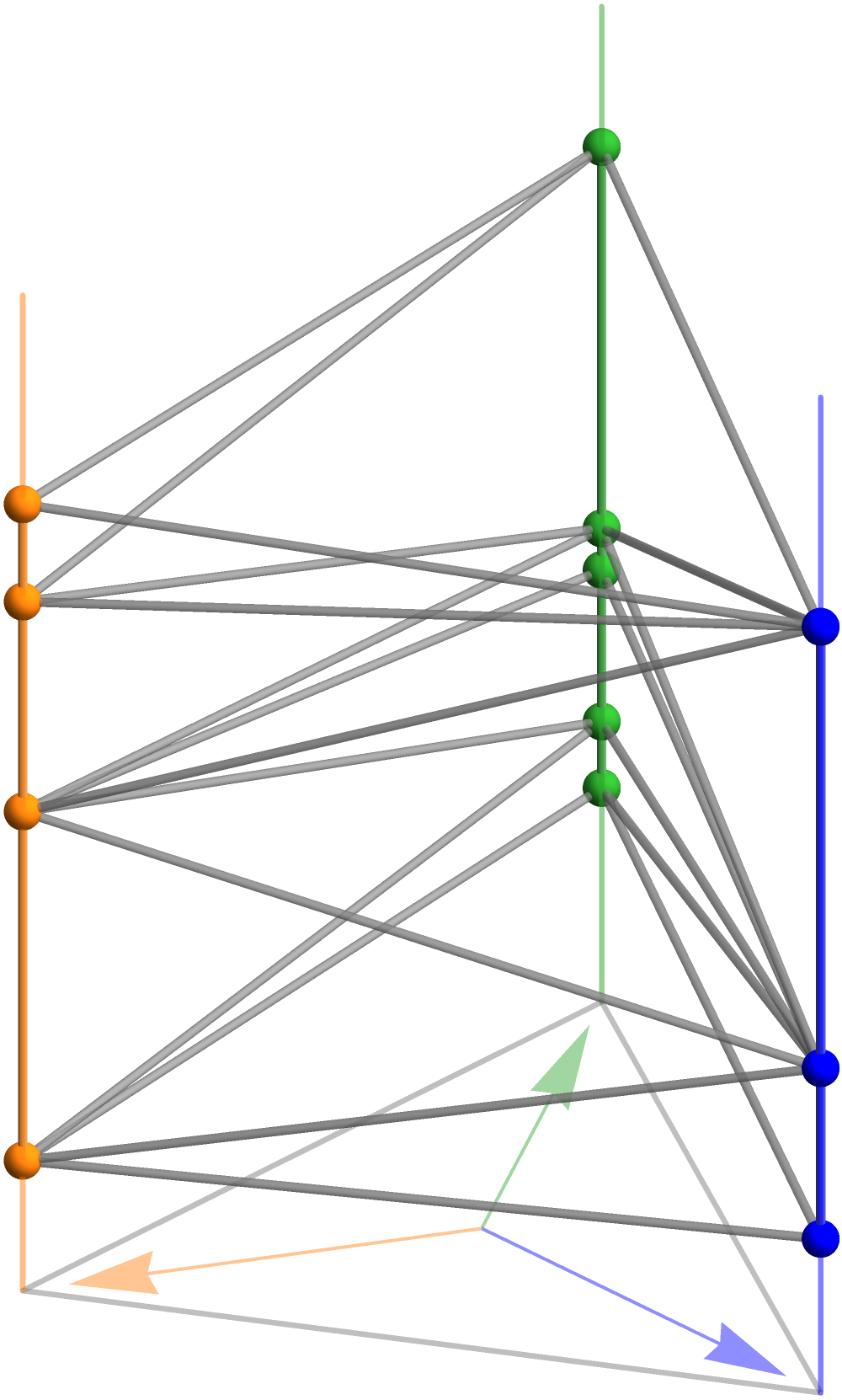}
    \hspace{20mm}
    \includegraphics[width=.22\textwidth]{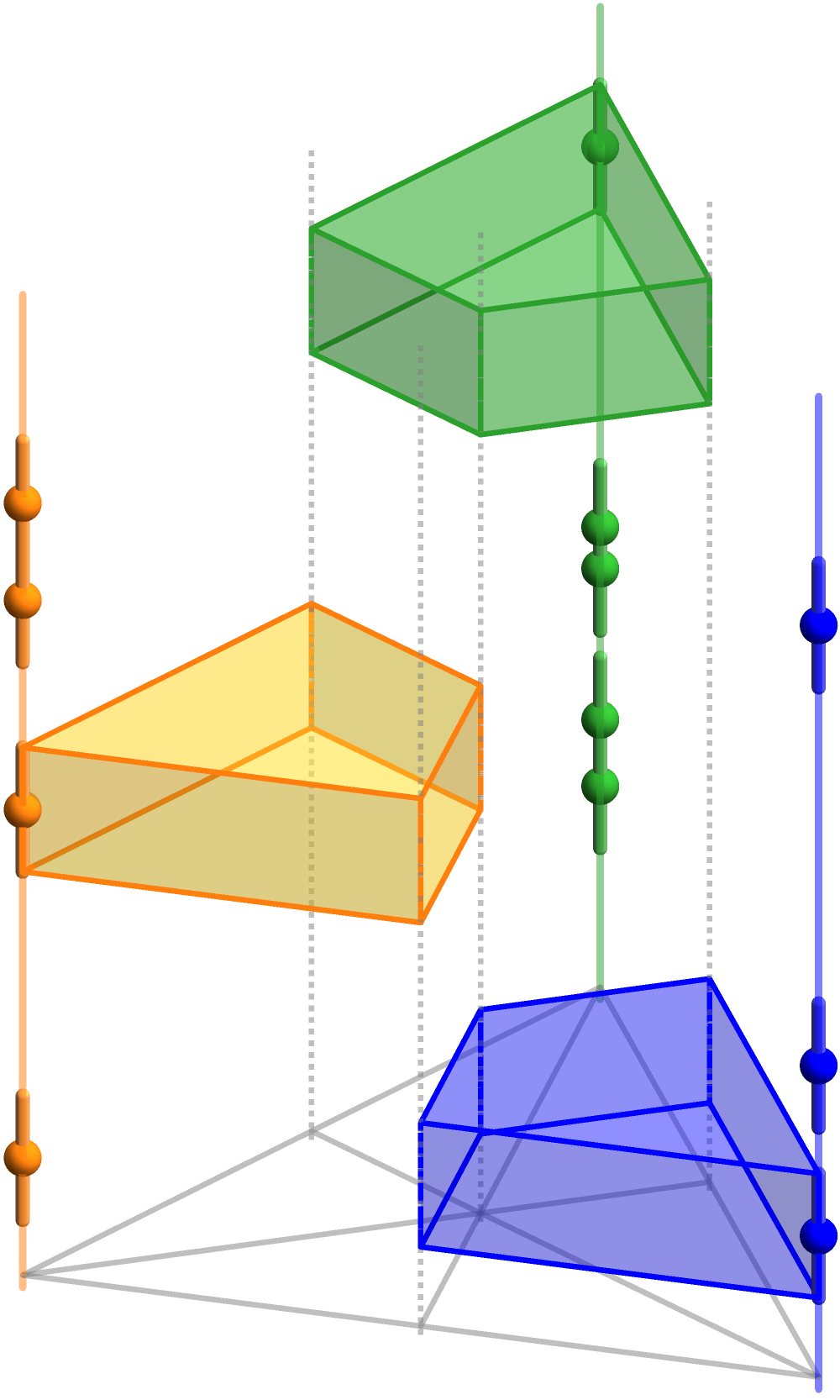}
    \caption{\footnotesize{The chromatic lifting for points in $\Rspace^1$ with three colors.
    \emph{Left:} the chromatic Delaunay complex embedded in three dimensions using the lifted points as vertices; it is isomorphic to the Delaunay complex of the lifted points.
    The triangles and tetrahedra are left unfilled for clarity.
    \emph{Right:} the union of segments of length $2r$, for each color separately along the three lines.
    Note the barycentric subdivision of the triangle at the \emph{bottom}, and the parallel lines emanating from the midpoints of the edges and the barycenter of the triangle.
    As defined shortly, each plank is a quadrangular prism that connects one of the intervals with its projections to three of these parallel lines.
    We show one such prism for each color.}}
    \label{fig:lifitng_1d_3_colors}
\end{figure}

\noindent \emph{Remark.}
The lifting construction provides an interesting insight independent of the focus of the paper.
In machine learning, categorical data is often vectorized via \emph{one-hot encoding}: the $j$-th out of $s+1$ categories is represented by the canonical vector $e_j\in\Rspace^{s+1}$ with one at the $j$-th place and zeros everywhere else.
With the minor modification of embedding the standard $s$-simplex in $\Rspace^{s+1}$ instead of $\Rspace^{s}$, the chromatic lifting can be interpreted as concatenating the spatial coordinate of each point with the one-hot encoding of its color. 
Corollary~\ref{cor:chromatic_delaunay_is_delaunay} implies that endowing this embedding with the Euclidean metric has concrete geometric meaning.

\subsection{Chromatic Subcomplexes}
\label{sec:3.4}

The inclusion of one color into all others, $\Alpha{r}{A_j} \hookrightarrow \Alpha{r}{\chi}$, has two major drawbacks:  it effectively uses only two colors ($j$ versus the rest), which is too little information to detect any of the tri-chromatic patterns in Figure~\ref{fig:patterns-filled}, and it is asymmetric by construction.
To overcome these issues, we call $\sigma$ the \emph{color simplex}, write $\Sigma$ for the complex of faces of $\sigma$, and for any subcomplex, $\Gamma \subseteq \Sigma$, define the \emph{$\Gamma$-subcomplex} of the chromatic alpha complex:
\begin{align}
  \AlphaSub{r}{\chi}{\Gamma} &= \left\{ \nu \in \Alpha{r}{\chi} \mid \chi(\nu) \in \Gamma \right\}.
\end{align}
For example, if $\Gamma$ consists of all subsets of size $t+1$ or less, then $\AlphaSub{r}{\chi}{\Gamma}$ is the collection of all simplices whose vertices have at most $t+1$ different colors, and we call this the \emph{$(t+1)$-chromatic subcomplex} of the alpha complex.
This choice of $\Gamma$ is symmetric, as it prefers no colors over any other colors.
For $t=0$, $\AlphaSub{r}{\chi}{\Gamma}$ is the disjoint union of the $s+1$ mono-chromatic alpha complexes, and we will see shortly that studying the inclusion $\AlphaSub{r}{\chi}{\Gamma} \hookrightarrow \Alpha{r}{\chi}$ is equivalent to studying the natural map
\begin{align}
  \left[ \bigcup\nolimits_{a \in A_0} \Ball{r}{a} \right] \sqcup \left[ \bigcup\nolimits_{a \in A_1} \Ball{r}{a} \right] \sqcup \ldots \sqcup \left[ \bigcup\nolimits_{a\in A_s} \Ball{r}{a} \right] ~~\longrightarrow~~ \bigcup\nolimits_{a\in A} \Ball{r}{a},
\end{align}
acting as inclusion on each color.
For example, this map captures the loops composed of points of \emph{any one} of the colors that are filled by points of the \emph{other} colors: when the homology functor is applied, such a loop becomes a non-trivial homology class that maps to zero.

\smallskip
We need definitions to gain intuition and give meaning to $\Gamma$-subcomplexes for $\Gamma$ more general than just the vertices in the color simplex.
We will make use of the chromatic lifting defined in the previous section, and instead of growing balls around the points, we grow what we call planks around the lifted points, which are balls in $\Rspace^d$ extruded into the $s$ extra color dimensions.

\smallskip
To begin, we fix a chromatic lifting of $\chi: A \rightarrow\sigma$, as in Section~\ref{sec:3.3}, with points $u_0, u_1, \ldots, u_s \in \Rspace^s$.
With a slight abuse of notation, we write $\Sigma$ for the simplicial complex that consists of all simplices spanned by any subset of these $s+1$ points, and $|\Sigma|$ for its underlying space.
The \emph{barycenter} of a simplex is the average of its vertices.
A \emph{chain} in $\Sigma$ is a nested sequence of its simplices, which gives a sequence of points (the barycenters), and taking their convex hull, we get a simplex in the \emph{barycentric subdivision} of $\Sigma$, denoted $\Sd{\Sigma}$; see Figure~\ref{fig:lifitng_1d_3_colors} where we see the barycentric subdivision of a triangle at the bottom of the right panel.
The \emph{star} of $u_j$ in $\Sd{\Sigma}$, denoted $\Star{j}{\Sd{\Sigma}}$, is the underlying space of all simplices in $\Sd{\Sigma}$ whose corresponding chains in $\Sigma$ contain only simplices that share $u_j$.
For example, the star of a vertex of a barycentrically subdivided triangle is the convex quadrangle that is the union of the two triangles spanned by the vertex, the barycenters of the two incident edges, and the barycenter of the triangle in $\Sigma$; see Figure~\ref{fig:lifitng_1d_3_colors}.
Note that a collection of vertex stars restricted to $|\Gamma|$ have a non-empty common intersection iff their colors form a simplex in $\Gamma$.
\begin{figure}[htb]
  \centering
  \raisebox{-.07\height}{\includegraphics[width=.25\textwidth]{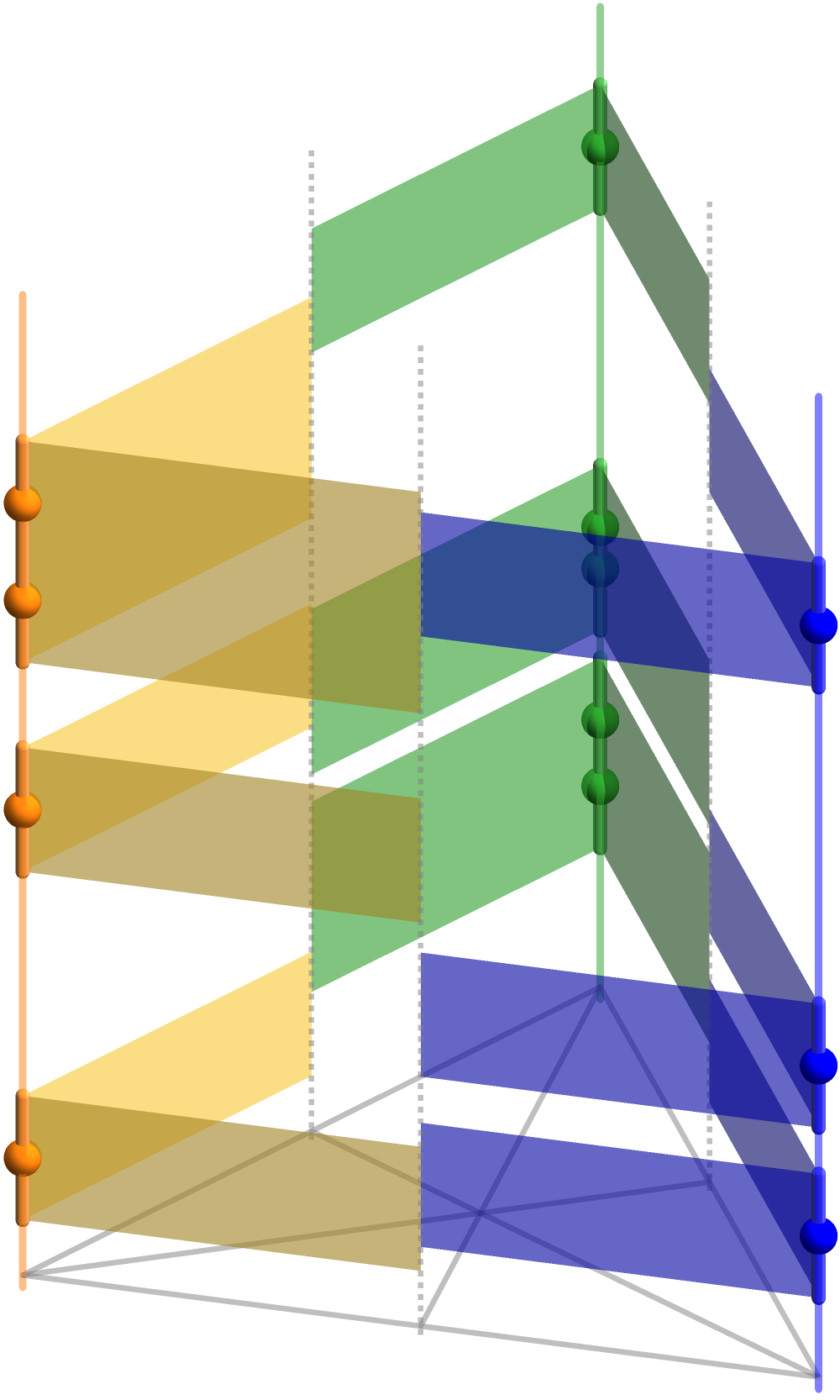}} \hspace{0.3in} 
  \raisebox{0\height}{\includegraphics[width=.62\textwidth]{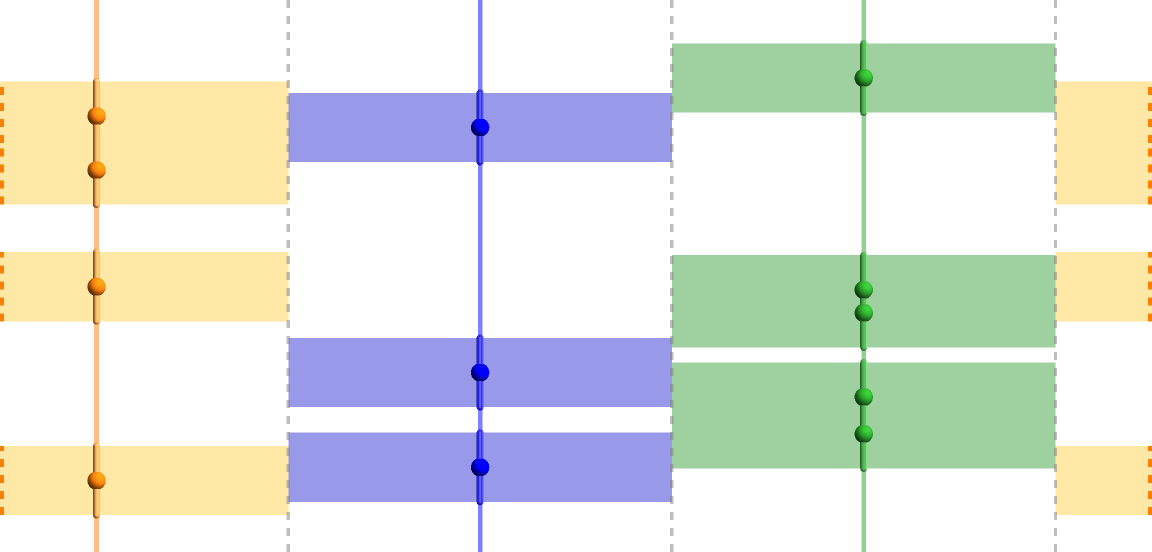}}
  \caption{\footnotesize The union of bi-chromatic planks for a tri-chromatic point set in $\Rspace^1$.
    It is homotopy equivalent to the bi-chromatic subcomplex of the corresponding chromatic alpha complex.
    On the \emph{left}, we show the planks in the sides of the triangular prism erected on top of the barycentrically subdivided color triangle.
    On the \emph{right}, the three sides of the triangular prism are unfolded into the plane, and the planks are glued along the \emph{orange dashed} lines.
    A similar unfolding one dimension higher helps us to understand the situation for the $2$-dimensional data in Figure~\ref{fig:bi-chromatic_subcomplex_2d_3_colors}.}
  \label{fig:bi-chromatic_subcomplex_1d_3_colors}
\end{figure}
\begin{figure}[htb]
  \centering
  \includegraphics[width=.95\textwidth]{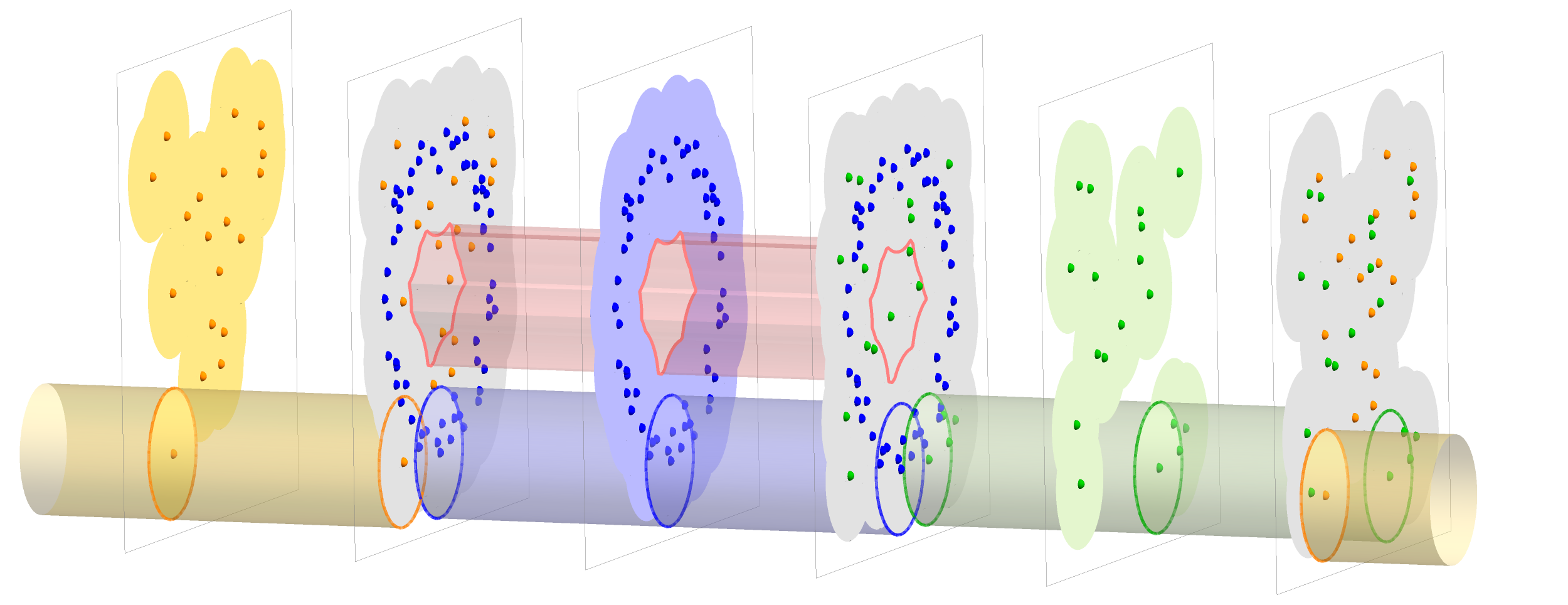}
  \caption{\footnotesize The union of bi-chromatic planks for a tri-chromatic point set in $\Rspace^2$.
    For clarity, only one plank per color is shown.
    Every blue, green; and yellow disk is connected to its gray counterparts via a solid cylinder.
    By construction, the planks are subsets of the boundary faces of a $4$-dimensional triangular prism.
    Similar to Figure~\ref{fig:bi-chromatic_subcomplex_1d_3_colors}, we unfold the $3$-dimensional boundary so we can illustrate the planks in $\Rspace^3$, as shown.
    Observe the highlighted $2$-hole in the middle:  this is a topological feature that captures a loop created by one color (\emph{blue}) and filled by each of the other two colors.
    This is one variant of the pattern of type~1+2 from Figure~\ref{fig:patterns-filled}.}
  \label{fig:bi-chromatic_subcomplex_2d_3_colors}
\end{figure}

\smallskip
With these notions, we are ready to define the \emph{planks}, which are instrumental to relate the union of balls with the subcomplexes of the chromatic alpha complex.
We have three progressively smaller variants:  the first for the entire $\Sigma$ (Figure~\ref{fig:lifitng_1d_3_colors}), the second restricted to a subcomplex $\Gamma \subseteq \Sigma$ (Figure~\ref{fig:bi-chromatic_subcomplex_1d_3_colors}), and the third further restricted to within the mono-chromatic Voronoi domains:
\begin{align}
  \Plank{r}{a}{\chi}           &= \Ball{r}{a} \times \Star{\chi(a)}{\Sd{\Sigma}}, 
    \label{eqn:plank1} \\
  \CPlank{r}{a}{\chi}{\Gamma}  &= \Ball{r}{a} \times \bigl( \Star{\chi(a)}{\Sd{\Sigma}} \cap |\Gamma| \bigr), 
    \label{eqn:plank2} \\
  \VCPlank{r}{a}{\chi}{\Gamma} &= \VBall{r}{a}{\chi} \times \bigl( \Star{\chi(a)}{\Sd{\Sigma}} \cap |\Gamma| \bigr).
    \label{eqn:plank3}
\end{align}
The result we are about to prove follows from the Nerve Theorem, which has a long history and many versions, which vary in assumptions and generality.
The historically first instances appeared in the papers by Leray~\cite{Ler45}, Borsuk~\cite{Bor48}, and Weil~\cite{Wei52}.
We use a recent version \cite[Thm B/3.11]{BKRR23}, which, in particular, also talks about diagrams induced by inclusions.
\begin{theorem}
  \label{thm:homotopy_equivalence_commutativity_general}
  Let $\chi \colon A \to \sigma$ be a chromatic point set, and $\Gamma \subseteq \Sigma$ a subcomplex of the color simplex.
  \begin{enumerate}
    \item For every radius $r$, the union of $\Gamma$-planks is homotopy equivalent to the $\Gamma$-subcomplex of the chromatic alpha complex:
    $\bigcup\nolimits_{a \in A} \CPlank{r}{a}{\chi}{\Gamma} \simeq \AlphaSub{r}{\chi}{\Gamma}$.
    \item \label{thm:homotopy_equivalence_commutativity_general:cd} %
    The homotopy equivalences commute with inclusions.
    Specifically, if $r \leq R$ and $\Gamma \subseteq \Delta \subseteq \Sigma$, then the following two diagrams commute:
    \[
      \begin{tikzcd}
        \bigcup\nolimits_{a \in A} \CPlank{R}{a}{\chi}{\Gamma} \arrow[r, leftrightarrow, "\simeq"] &
            \AlphaSub{R}{\chi}{\Gamma} \\
            \bigcup\nolimits_{a \in A} \CPlank{r}{a}{\chi}{\Gamma} \arrow[u, hook] \arrow[r, leftrightarrow, "\simeq"] &
            \AlphaSub{r}{\chi}{\Gamma} \arrow[u, hook]
        \end{tikzcd}\hspace{8mm}
        \begin{tikzcd}
            \bigcup\nolimits_{a \in A} \CPlank{r}{a}{\chi}{\Delta} \arrow[r, leftrightarrow, "\simeq"] &
            \AlphaSub{r}{\chi}{\Delta} \\
            \bigcup\nolimits_{a \in A} \CPlank{r}{a}{\chi}{\Gamma} \arrow[u, hook] \arrow[r, leftrightarrow, "\simeq"] &
            \AlphaSub{r}{\chi}{\Gamma} \arrow[u, hook]
        \end{tikzcd}
    \]
    \end{enumerate}
\end{theorem}
\begin{proof}
    Each union of planks of the type defined in \eqref{eqn:plank2} remains unchanged if we replace them by planks of the type defined in \eqref{eqn:plank3}. 
    We claim that for a fixed $r > 0$ and $\Gamma \subseteq \Sigma$, the nerve of the latter type of planks is isomorphic to the $\Gamma$-subcomplex of the chromatic alpha complex:
    \begin{align}
    \Nerve{\{ \VCPlank{r}{a}{\chi}{\Gamma} \mid a \in A \}} \cong \AlphaSub{r}{\chi}{\Gamma}.
    \end{align}
    Indeed, $\nu \subseteq A$ has intersecting such planks iff their Voronoi balls intersect and the stars of the colors restricted to $|\Gamma|$ intersect.
    As argued before Lemma~\ref{lem:homotopy_equivalence}, the former happens iff $\nu \in \Alpha{r}{\chi}$, and as mentioned earlier in this subsection, the latter happens iff $\chi(\nu) \in \Gamma$.
    Together, these are the two defining conditions for $\nu \in \AlphaSub{r}{\chi}{\Gamma}$.
    Provided the assumptions for the Nerve Theorem \cite[Thm B]{BKRR23} are satisfied, we thus have
    \begin{align}
    \bigcup\nolimits_{a \in A} \CPlank{r}{a}{\chi}{\Gamma} 
      &= \bigcup\nolimits_{a \in A} \VCPlank{r}{a}{\chi}{\Gamma} 
      \simeq \Nerve{\{ \VCPlank{r}{a}{\chi}{\Gamma} \mid a \in A \}} \cong \AlphaSub{r}{\chi}{\Gamma},
    \end{align}
    and these homotopy equivalences commute with inclusions.
    It remains to show that the two assumptions for the Nerve Theorem \cite[Thm B]{BKRR23} are indeed satisfied: ($a$) the sets are closed and convex, and ($b$) every non-empty intersection of sets contains a point that is preserved in all inclusions.

    \smallskip
    The planks $\VCPlank{r}{a}{\chi}{\Gamma}$ are not necessarily convex, but we can replace each by its convex hull in $R^{d+s}$ without changing any common intersection of two or more of them.
    To satisfy ($b$), we need a point in every non-empty common intersection of planks that is preserved by all relevant inclusions in the claimed diagrams.
    For $\VCPlank{r}{a}{\chi}{\Gamma}$ we take the point $\Lift{p}{a} = a + u_{\chi(a)}$.
    More generally, for a collection, $\nu \subseteq A$, we find the radius, $r$, and the point, $p_\nu$, such that $\bigcap_{a \in \nu} \VBall{r}{a}{\chi} = \{p_\nu\}$, and we set $\Lift{p}{\nu}$ equal to $p_\nu$ plus the barycenter of the simplex spanned by the $a_j$ with $j \in \nu$.
    Since $\Lift{p}{\nu}$ is the first point that appears in the common intersection of the growing planks, $\Lift{p}{\nu}$ is also contained in the common intersection if we substitute $R \geq r$ for $r$ or $\Delta \supseteq \Gamma$ for $\Gamma$.
\end{proof}

A careful analysis of the proof of \cite[Thm B]{BKRR23} shows that in our case we also get commuting homotopy equivalences for quotient spaces. 
Viewing the vertical inclusions in (\ref{thm:homotopy_equivalence_commutativity_general:cd}) as topological pairs, the horizontal maps are homotopy equivalences of pairs---that is, the homotopies preserve the subspaces. 
This implies that for $\Gamma\subseteq\Delta$ we can equivalently study the filtration of topological pairs $\left(\bigcup\nolimits_{a \in A} \CPlank{r}{a}{\chi}{\Delta},\  \bigcup\nolimits_{a \in A} \CPlank{r}{a}{\chi}{\Gamma}\right)$ and $\left(\AlphaSub{r}{\chi}{\Delta},\  \AlphaSub{r}{\chi}{\Gamma}\right)$. The same holds for the corresponding quotient spaces, which is particularly relevant for the study of the relative persistent homology in Section~\ref{sec:5}. For more details, see Appendix~\ref{apx:he_pairs}.

\subsection{Algorithm for Chromatic Alpha Complex}
\label{sec:3.5}

This section discusses how to compute the chromatic alpha complex. 
The procedure has two parts: first the construction of the chromatic Delaunay complex, and second the computation of the radius function on this complex. 
Following Corollary~\ref{cor:chromatic_delaunay_is_delaunay}, the chromatic Delaunay complex is computed as a standard Delaunay complex of the chromatic lifting of the input chromatic point set;
see \cite{CGAL17} for fast and widely available code and \cite{BCDES22} for bounds on the complexity for chromatic point sets.\footnote{Some implementations do not allow many points lying on the same affine subspace, and for others it slows down the computations.
An alternative is to slightly perturb the lifted points, compute the Delaunay complex, and then only keep the down-set of the maximal simplices that span all colors.}
Below we describe the computation of the radius function, and argue that for generic data (see Definition~\ref{dfn:conventional_genericity}) in fixed dimension and with a constant number of colors, the algorithm takes time linear in the size of the chromatic Delaunay complex.

\smallskip
For the algorithm, it is convenient to work with the squared radius, which we use for the remainder of this section. 
We fix a chromatic point set, $\chi \colon A\to \sigma$, with $A \subseteq \Rspace^d$ and $\card{\sigma}=s+1$. 
Let $\nu\in\Delaunay{}{\chi}$ be a simplex, $\tau=\chi(\nu)$ its colors, and $P = \bigcap_{v\in\nu} \domain{v}{A_{\chi(v)}}$ the intersection of the chromatic Voronoi domains of its vertices. 
Then the squared radius of the smallest empty stack that passes through the points in $\nu$---the value of $\nu$ under the squared radius function---is the result of convex optimization:
\begin{equation}\label{eq:min_max_radius}
    \Radiusf^2(\nu) = \min_{x \in P} \max_{v\in\nu} \|x-v\|^2.
\end{equation}
We write the convex optimization as an algorithm that avoids the explicit construction of $P$ and can be implemented using exact arithmetic.
For each $j \in \tau$, consider the affine subspace $E_j \subseteq \Rspace^d$ consisting of the points $x$ equidistant to all points in $\nu$ with color $j$, and the function $e_j \colon E_j \to \Rspace$ that maps $x \in E_j$ to the squared distance to any point of $\nu$ with color $j$. 
Observe that the intersection of these affine subspaces, $E = \bigcap_{j \in \tau} E_j$, is the smallest affine subspace that contains $P$.
The pointwise maximum function, $e \colon E \to \Rspace$ defined by $e(x) = \max_{j \in \tau} e_j(x)$, is a strictly convex function with unique minimum, $y \in E$.
If $y \in P$, then $e(y) = \min_{x \in P} e(x)$, so $\Radiusf^2(\nu)=e(y)$.
Otherwise, $\min_{x \in P}e(x)$ is attained on the boundary of $P$, which implies that $\Radiusf^2(\nu)$ is the smallest $\Radiusf^2(\mu)$ over all cofaces $\mu$ of $\nu$ in $\Delaunay{}{\chi}$.
Note that to query whether $y\in P$, we only need to check whether the stack centered in $y$ that passes through $\nu$ is empty, which is easy.

\smallskip
To formalize the algorithm, we write $S(x,r)$ for the $(d-1)$-sphere with radius $r$ and center $x \in \Rspace^d$.
The algorithm visits the simplices of $\Delaunay{}{\chi}$ in the order of decreasing dimension:
\begin{tabbing}
  m\=m\=m\=m\=m\=m\=m\=m\=m\=m\=\kill
  \> \> {\tt for} $p=s+d$ {\tt downto} $0$  {\tt do} \\*
  \> \> \> {\tt for each} $p$-simplex $\nu \in \Delaunay{}{\chi}$ {\tt do} \\*
  \> \> \> \> {\sc Step 1:} \= construct the affine spaces $E_j$, $j \in \chi(\nu)$, and $E = \bigcap_{j \in \chi(\nu)} E_j$; \\*
  \> \> \> \> {\sc Step 2:} \> construct $e_j \colon E_j \to \Rspace$, for each $j \in \chi(\nu)$, and $e \colon E \to \Rspace$; \\*
  \> \> \> \> {\sc Step 3:} \> find the unique minimum of $e$, the point $y \in E$; \\*
  \> \> \> \> {\sc Step 4:} \> {\tt if} \= $S(y, \sqrt{e_j(y)})$ is empty of $A_j$, for each $j \in \chi(\nu)$ \\*
  \> \> \> \> \> \> {\tt then} $\Radiusf(\nu) = e(y)$ \\*
  \> \> \> \> \> \> {\tt else} $\Radiusf(\nu) = \min \{ \Radiusf(\mu) \mid \nu \subseteq \mu, \mu\in\Delaunay{}{\chi}, \dime{\mu} = p+1\}$ \\*
  \> \> \> \> \> {\tt endif} \\*
  \> \> \> {\tt endfor} \\*
  \> \> {\tt endfor}
\end{tabbing}
In particular, for $p=0$ the algorithm sets $\Radiusf (a) = 0$ for every vertex $a \in \Delaunay{}{\chi}$.
Assume now that the chromatic lifting of the points is generic; see Definition~\ref{dfn:conventional_genericity}.
Assuming constant $d$ and $s$, every step takes only constant time, except Step~4, which loops over cofaces both for checking emptiness and for determining the coface with smallest squared radius.
We will see shortly that Step~4 takes constant time in the amortized sense.
\begin{theorem}
  \label{thm:chromatic_radius_function_in_linear_time}
  Let $\chi \colon A \to \Rspace$ be a chromatic point set, with $A \subseteq \Rspace^d$ finite and generic, set $s = \dime{\sigma}$, and let $m$ be the number of simplices in $\Delaunay{}{\chi}$.
  Assuming $d$ and $s$ are constants, $\Radiusf \colon \Delaunay{}{\chi} \to \Rspace$ can be computed in $O(m)$ time.
\end{theorem}
\begin{proof}
  The body of the algorithm is executed once for each simplex $\nu\in\Delaunay{}{\chi}$.
  It is easy to see that Steps~1 and 2 take only constant time each.
  To see the same for Step~3, we observe that $y$ is the center of the smallest sphere that encloses all vertices of $\nu$ and whose center lies in $E$.
  The latter condition can be enforced by reflecting $\nu$ through $E$ and adding its vertices to the points to be enclosed.
  The number of points to be considered is at most $2(d+s+1) = O(1)$, so we can compute the smallest enclosing sphere in constant time with the miniball algorithm \cite{Wel91} or indeed a brute-force algorithm that checks all possibilities.
  In Step~4 we loop through the cofaces of $\nu$, both for checking emptiness of stacks, and to find the minimum radius in the \texttt{else}-clause.
  There can be many such cofaces for any individual $\nu$, but any $(p+1)$-simplex $\mu \in \Delaunay{}{\chi}$ is a coface of only $p+2$ $p$-simplices.
  Since $p+2 \leq d+s+1$, this implies that in total we run at most $O\big((d+s+1)\cdot m\big) = O(m)$ tests.
\end{proof}

\section{Chromatic Radius Functions are Generalized Discrete Morse}
\label{sec:4}

As before, we assume that $\chi \colon A \to \sigma$ is a chromatic point set with finite $A \subseteq \Rspace^d$ and $\sigma = \{0,1,\ldots,s\}$, and we write $\Radiusf \colon \Delaunay{}{\chi} \to \Rspace$ for the chromatic radius function.
By Definition~\ref{dfn:generalised_discrete_morse_function}, we call $\Radiusf$ a generalized discrete Morse function if every level set is a union of disjoint maximal intervals.
The purpose of this section is to prove that $\Radiusf$ is indeed generalized discrete Morse, provided $A$ satisfies a genericity condition, which we introduce first.

\subsection{Chromatic Genericity}
\label{sec:4.1}

Definition~\ref{dfn:conventional_genericity} requires from a mono-chromatic set that any $p$-sphere passes through at most $p+2$ of its points.
We strengthen this requirement so it can be applied to the chromatic case.
\begin{definition}[Chromatic Genericity]
  \label{dfn:chromatic_genericity}
  Call a finite set $A \subseteq \Rspace^d$ \emph{chromatically generic} if every $k+1$ concentric $(d-1)$-spheres pass through at most $d+k+1$ points, and their intersections with any affine $p$-plane pass through at most $p+k+1$ of these points.
\end{definition}
Letting $k+1$ be the number of spheres in a stack, and $m$ the number of points on these spheres, we sometimes call $m-(k+1)$ the \emph{surplus} of the configuration.
Definition~\ref{dfn:chromatic_genericity} limits the surplus to $d$, or to $p$, respectively.
We will use two equivalent formulations of chromatic genericity.
To formulate them, write $E(B)$ for the maximal affine subspace whose points are equidistant to all points in $B$.
It is the common intersection of all bisecting hyperplanes of any two points in $B$.
\begin{lemma}\label{lem:equivalent_genericity_conditions}
  Assume $A \subseteq \Rspace^d$ is finite. 
  Then the following two conditions are equivalent to $A$ being chromatically generic:
  \begin{enumerate}[(a)]
    \item if $B_0, B_1, \dots, B_k$ are non-empty disjoint subsets of $A$, then either $E = E(B_0) \cap E(B_1) \cap \ldots \cap E(B_k)$ is empty or the codimension of $E$ is equal to the surplus, namely $\sum_{j=0}^k \card{B_j} - (k+1)$;
    \item if $S_0, S_1, \dots, S_k$ are concentric spheres, $B_j = S_j \cap A$, and $c_j \in B_j$ is an arbitrary but fixed choice for $0 \leq j \leq k$, then the vectors $\{ b-c_j \mid b \in B_j \setminus \{c_j\},\, 0 \leq j \leq k \}$ are linearly independent.
  \end{enumerate}
\end{lemma}
\begin{proof}
  We establish the equivalences by showing that the chromatic genericity of $A$ implies (a), that (a) implies (b), and that (b) implies the chromatic genericity of $A$.

  \smallskip
  Chromatic genericity $\Rightarrow$ (a).
  We show the contrapositive. Let $B_0, B_1, \ldots, B_k$ be non-empty disjoint subsets that violate (a), and suppose that they minimize the surplus among all such violating collections.
  To violate (a), $E$ must be non-empty and at least one of the sets must contain more than one point.
  Suppose $\card{B_0} \geq 2$, let $x \in B_0$, write $B_0' = B_0 \setminus \{x\}$, and note that (a) holds for $B_0', B_1, \ldots, B_k$, by extremal assumption.
  The surplus of the latter collection is $\sum_{j=0}^k \card{B_j} - k$, which is therefore the codimension of $E' = E(B_0') \cap E(B_1) \cap \ldots \cap E(B_k)$.
  It is also the codimension of $E$, since $E$ is contained in $E'$ and the codimensions differ by at most one.
  It follows that the two spaces coincide.
  Write $B' = B_0' \sqcup B_1 \sqcup \ldots \sqcup B_k$, let $H$ be the affine hull of $B$, and set $p = \dime{H}$.
  Let $t$ be the smallest number for which there exist $t+1$ concentric spheres, $S_0, S_1, \ldots, S_t$ in $H$, passing through each of $B_0', B_1, \ldots, B_k$ (a sphere may pass through more than one subset). 
  We have $t \leq k$ since we may choose the common center of the spheres in $E'$.
  We claim $\card{B'} \geq p+t+1$.
  Assuming $\card{B'} < p+t+1$, the codimension of $E'$ is less than $p$, so $\dim E' + \dim H > d$, which implies the existence of a line, $L$, common to $H$ and $E'$.
  We can therefore find indices $0 \leq i < j \leq t$ and points $y \in S_i$ and $z \in S_j$ such that the bisector of $y, z$ intersects $L$ in a point, $o$.
  Choosing the common center of the spheres at $o$, we thus get only $t$ spheres, which contradicts the choice of $t$.
  So $\card{B'} \geq p+t+1$, as claimed.
  Since the constructed stack of spheres is centered at a point in $E' = E\subseteq E(B_0)$, the sphere passing through $B'_0$ also passes through $x$.
  We thus have a stack of $t+1$ spheres in $H$ that passes through more than $p+t+2$ points, which shows that $A$ is not chromatically generic.

  \smallskip
  (a) $\Rightarrow$ (b).
  Assume $S_j, B_j, c_j$ are as in (b).
  For each $j$, we write $U_j = \{b-c_j \mid b \in B_j \setminus \{c_j\} \}$ and note that $E(B_j)$ is a translate of the orthogonal complement of $\Span{U_j}$.
  Writing $U = U_0 \cup U_1 \cup \ldots \cup U_k$, we thus get $E$ as a translate of the orthogonal complement of $\Span{U}$.
  By (a), the codimension of $E$ is $\sum_{j=0}^k \card{B_j} - (k+1)$, which is therefore the dimension of $\Span{U}$.
  But this is also the number of vectors in $U$, which implies that they are linearly independent, as required to get (b).

  \smallskip
  (b) $\Rightarrow$ chromatic genericity.
  Let $m = \sum_{j=0}^k \card{B_j}$.
  By (b), we get $m - (k+1)$ linearly independent vectors in $\Rspace^d$.
  Therefore $m \leq d+k+1$.
  If all $m$ points lie in a $p$-dimensional affine subspace, then the dimension of the span of the vectors is at most $p$, which implies $m \leq p+k+1$, as required for the chromatic genericity.
\end{proof}

We call the requirement in Definition~\ref{dfn:chromatic_genericity} a genericity condition, tacitly implying that it is satisfied by almost all finite sets.
We prove that this is indeed the case.
In the argument, we concatenate the coordinates of $n$ points in $\Rspace^d$ so we can think of the set as a point in $\Rspace^{nd}$.
\begin{lemma}\label{lem:chromatic_genericity_is_generic}
  For each positive integer, $n$, the family of sets of $n$ points in $\Rspace^d$ that violate chromatic genericity is a finite union of sets with measure zero in $\Rspace^{nd}$.
\end{lemma}
\begin{proof}
  It suffices to consider the $d$-dimensional condition in Definition~\ref{dfn:chromatic_genericity}.
  Indeed, a configuration that violates the $p$-dimensional condition implies at least $p+2$ points in an affine $p$-plane, and the sets that contain such $p+2$ points belong to a subset of dimension at most $nd-1$ of $\Rspace^{nd}$.
  Consider $k+1$ concentric spheres and $d+k+2$ points on these spheres in $\Rspace^d$.
  The surplus of this configuration is $d+1$, which violates chromatic genericity.
  Assigning the sum of squares of the $d$ coordinates as a $(d+1)$-st coordinate to each point, we get $d+k+2$ points on $k+1$ parallel hyperplanes in $\Rspace^{d+1}$.
  For each hyperplane, pick one of its points and take the difference vectors to the other points on this hyperplane.
  This gives a total of $(d+k+2) - (k+1) = d+1$ linearly dependent vectors in $\Rspace^{d+1}$.
  Writing $(x_{i,1}, x_{i,2}, \ldots, x_{i,d+1})$ for the $i$-th vector, the $d+1$ vectors satisfy
  \begin{align*}
    \det \left[ \begin{array}{cccc}
                x_{1,1} & x_{1,2} & \ldots & x_{1,d+1} \\
                x_{2,1} & x_{2,2} & \ldots & x_{2,d+1} \\
                \vdots  & \vdots  & \ddots & \vdots    \\
                x_{d+1,1} & x_{d+1,2} & \ldots & x_{d+1,d+1}
             \end{array} \right]  &=  0.
  \end{align*}
  This is a polynomial in the coordinates of $\Rspace^{nd}$ that is not everywhere zero.
  Hence, its zero-set is a subspace of dimension at most $nd-1$.

  We have such a polynomial for any $d+k+2$ points and their partition into $k+1$ sets. This is a finite collection as $k$ is bounded by $n$.
  For a set of $n$ points in $\Rspace^d$ to be chromatically generic, it suffices to avoid the resulting finite number of zero-sets, each of dimension at most $nd-1$, which implies Lebesgue measure zero in $\Rspace^{nd}$.
\end{proof}

\subsection{Convex Optimization}
\label{sec:4.2}

The approach mimics the proof in \cite[Section~4]{BaEd17} that the radius function on the Delaunay complex of a mono-chromatic point set is generalized discrete Morse.
As before, $\chi \colon A \to \sigma$ is a chromatic set of finitely many points with $s+1$ colors in $\Rspace^d$. 
Given a collection of points, $\nu\in\Delaunay{}{\chi}$, we define the \emph{smallest empty circumstack} as the solution to an optimization problem with variables $z\in\Rspace^d$ for the center of the stack, and $r=(r_0, r_1, \dots, r_s)$ for the radii of the spheres in the stack:
\begin{align*}
    \operatornamewithlimits{minimize}\limits_{z,r} \hspace{5mm} & \max \{ r_0, r_1, \dots, r_s \} , \\
    \text{subject to} \hspace{5mm} & \norm{x - z} = r_{\chi(x)} && \text{for $x\in\nu$},\\
    & \norm{x - z} \geq r_{\chi(x)} && \text{for $x\in A\setminus \nu$}.
\end{align*}
We want to turn this into a differentiable convex optimization problem. 
Since the maximum function is not differentiable, we introduce a new variable, $b$. 
The constraints need to be either inequalities, $g \leq 0$, for convex differentiable $g$, or equalities or inequalities, for affine $g$. 
We switch to squared distances and radii so that $g_x(z) = \norm{x-z}^2 - r_{\chi(x)}^2$ is differentiable and strictly convex, but the inequalities are in the wrong direction.
We substitute new variables, $a = (a_0, a_1, \dots, a_s)$, for the radii, constrain $b$ to be smaller than or equal to all $a_i$, and get an equivalent optimization problem, whose constraints are affine with respect to the new variables:
\begin{align*}
    a_j &= \norm{z}^2 - r_j^2 
    \hspace{5mm} \text{for all $j=0, 1, \dots, s$}, \\
    g_x(z,a) \,&=\, \norm{x-z}^2 - r_{\chi(x)}^2 \,=\, \norm{x}^2 - 2\scalprod{x}{z} + \norm{z}^2 - r_{\chi(x)}^2 \,=\, \norm{x}^2 -  2\scalprod{x}{z} + a_{\chi(x)}.
\end{align*}
For any $\nu\in\Delaunay{}{\chi}$, we thus get the following differentiable convex optimization problem, $\problemP{\nu}$, in which we write $h_j (a,b) = b-a_j$:
\begin{align*}
    \operatornamewithlimits{minimize}\limits_{z,a,b} \hspace{5mm} & f(z,b) = \norm{z}^2 - b , \span \span \\
    \text{subject to} \hspace{5mm} & \text{$h_j(a,b) \leq 0$} && \text{for $j= 0,1,\ldots, s$}, \\
    & g_x(z,a) = 0 && \text{for $x\in\nu$}, \\ 
    & g_x(z,a) \geq 0 && \text{for $x\in A\setminus \nu$}.
\end{align*}
We formulate the task this way in order to make use of duality, which is a powerful tool in convex optimization \cite[Chapter~5.2]{BoVa04}. 
The \emph{Lagrange dual problem} of $\problemP{\nu}$, denoted $\problemD{\nu}$ and with variables $\lambda = (\lambda_x)_{x\in A}$ and $\mu = (\mu_j)_{j\in\sigma}$, is the following:
\begin{align*}
  \operatornamewithlimits{maximize}\limits_{\lambda, \mu} \hspace{5mm} & G(\lambda, \mu) = \inf\limits_{z,a,b} f(z,b) + \sum\nolimits_{x\in A} \lambda_x g_x(z, a) + \sum\nolimits_{j\in\sigma} \mu_j h_j(a, b), \span \span \\
  \text{subject to} \hspace{5mm} & \lambda_x \leq 0 && \text{for $x\in A\setminus \nu$}, \\
  & \mu_j \geq 0 && \text{for $j\in\sigma$}.
\end{align*}
Because the sums in $\problemD{\nu}$ are both non-positive,
the value of the dual problem for a feasible solution is always smaller than or equal to the value of the primal problem for any of its feasible solutions.
This is in particular true for the optimal values, and the difference between those is referred to as the \emph{optimal duality gap}. 
Under the chromatic genericity conditions of Definition~\ref{dfn:chromatic_genericity}, the gap for problem $\problemP{\nu}$ is guaranteed to be zero and attained by some $z, a, b, \lambda, \mu$.
This is because $\problemP{\nu}$ is convex and satisfies the Slater's condition \cite[Section~5.2.3]{BoVa04}, which states that there exists a feasible solution with all inequalities strict. 
For $\problemP{\nu}$, this translates to the claim that if there exists an empty circumstack of $\nu$, then there also exists one that passes through no points from $A\setminus\nu$. This reformulation is implied by Lemma~\ref{lem:equivalent_genericity_conditions}~(a).
Indeed, if spaces of equidistant points intersect generically, then so do Voronoi cells, and the desired stack can be centered at any point in the interior of $\bigcap_{v\in\nu}\domain{v}{A_{\chi(v)}}$.

\smallskip
Since $\problemP{\nu}$ is strictly convex as well as differentiable, points $z, a, b, \lambda, \mu$ are primal and dual optima iff they satisfy the \emph{Karush--Kuhn--Tucker} (KKT) conditions \cite[Section~5.5.3]{BoVa04}:
\begin{align}
    & \text{$z, a, b$ is primal feasible, and $\lambda, \mu$ is dual feasible} \label{eq:kkt:feasibility}, \\
    & \text{$\lambda_x \cdot g_x(z,a) = 0$ for all $x\in A$, and $\mu_j \cdot h_j(a,b) = 0$ for all $j\in\sigma$}, \label{eq:kkt:slackness} \\
    & \text{$\nabla f(z,b) + \sum\nolimits_{x\in A} \lambda_x \nabla g_x(z,a) + \sum\nolimits_{j\in\sigma} \mu_j \nabla h_j(a,b) = 0$}. \label{eq:kkt:gradient}
\end{align}
Note that the second condition implies that $\lambda_x \neq 0$ only if $g_x(z,a)=0$, i.e., the stack passes through~$x$. 
Similarly, $\mu_j\neq 0$ only if the $j$-colored sphere has maximum radius among the spheres on the stack.
Next, we give the gradients at point $(z,a,b)$ needed in the last condition:
\begin{align*}
  \nabla f &= (2z; 0, \dots, 0; -1), \\
  \nabla g_x &= (-2x; 0,\dots,0,1,0,\dots,0;0) \hspace{5mm} \text{with $1$ at the position corresponding to $a_{\chi(x)}$}, \\
  \nabla h_j &= (0; 0,\dots,0,-1,0,\dots,0;1) \hspace{7.1mm} \text{with $-1$ at the position corresponding to $a_{j}$}.
\end{align*}
Putting all the above observations together---the existence of zero gap solutions, the KKT conditions, and the gradients---we get the following:
\begin{lemma}\label{lem:kkt} 
  Let $\chi \colon A \to \sigma$, with $A \subseteq \Rspace^d$ chromatically generic, and $\nu \in \Delaunay{}{\chi}$.
  Let $z, a, b$ describe an empty circumstack of $\nu$, and let $\eta \supseteq \nu$ contain all points of $A$ that lie on this stack.
  Then this is the smallest empty circumstack of $\nu$ iff there exist $\lambda_x$ for $x \in A$ and $\mu_j$ for $j \in \sigma$ such that:
  \begin{align}
    &\sum\nolimits_{x\in\eta} \lambda_x x = z ,\label{eq:kkt:z} \\
    &\sum\nolimits_{x\in\eta_j} \lambda_x = \mu_j \text{, in which $\eta_j=\eta\cap\chi^{-1}(j)$} ,\label{eq:kkt:lambda} \\
    &\sum\nolimits_{j\in\sigma} \mu_j = 1 , \label{eq:kkt:mu} \\
    &\lambda_x \leq 0 \text{ for $x \in A \setminus \nu$, and } \lambda_x = 0 \text{ if } x \in A \setminus \eta, \label{eq:kkt:lambdaleq} \\
    &\mu_j \geq 0 \text{ for all $j\in\sigma$, and } \mu_j = 0 \text{ if the $j$-th sphere does not have maximum radius.} \label{eq:kkt:mugeq}
  \end{align}
\end{lemma}
\begin{proof}
  We argue the five conditions in reverse order, from \eqref{eq:kkt:mugeq} to \eqref{eq:kkt:z}.
  The inequality in \eqref{eq:kkt:mugeq} just rewrites the second condition in $\problemD{\nu}$, and the strengthening equality is implied by the second slackness condition in \eqref{eq:kkt:slackness}.
  Similarly, the inequality in \eqref{eq:kkt:lambdaleq} just rewrites the first condition in $\problemD{\nu}$, and the strengthening equality is implied by the first slackness condition in \eqref{eq:kkt:slackness}.
  For the remaining three conditions, we plug the gradients into \eqref{eq:kkt:gradient}:
  \begin{align}
    (2z;0,\ldots,0;-1) + \sum\nolimits_{x \in A} \lambda_x (-2x;0,\ldots,1,\ldots,0;0) + \sum\nolimits_{j \in \sigma} \mu_j (0; 0,\ldots,-1,\ldots,0; 1) &= 0.
  \end{align}
  Comparing the last coordinates, we get \eqref{eq:kkt:mu}, and comparing the coordinates that correspond to color $j$, we get \eqref{eq:kkt:lambda}.
  Combining \eqref{eq:kkt:lambda} and \eqref{eq:kkt:mu}, we get $\sum_{x \in A} \lambda_x = \sum_{x \in \eta} \lambda_x = 1$ because $\lambda_x = 0$ if $x \in A \setminus \eta$.
  Now comparing the first coordinates of the gradients, we get \eqref{eq:kkt:z}.
\end{proof}
Since \eqref{eq:kkt:lambda} and \eqref{eq:kkt:mu} imply $\sum_{x \in \eta} \lambda_x = 1$, the dual solution expresses the center of the stack as an affine combination of the points on the spheres in this stack; see \eqref{eq:kkt:z}.
It pays to unpack this interpretation by distinguishing the spheres with maximum radius from the others.
Let $\sigma' \subseteq \sigma$ be the colors whose spheres have maximum radius and write $\sigma'' = \sigma \setminus \sigma'$.
For each $i \in \sigma''$, we have $\mu_i = 0$, so the sum of the corresponding $\lambda_x$ vanishes, which suggests we interpret the corresponding combination as a vector:
$\sum\nolimits_{x \in \eta_i} \lambda_x x = \sum\nolimits_{x \in \eta_i} \lambda_x (x - y)$,
in which $y = y(i)$ is an arbitrary but fixed point in $\eta_i$.
For each $j \in \sigma'$, the corresponding combination of the points is $\sum_{x \in \eta_j} \lambda_x x$.
With this, we can rewrite \eqref{eq:kkt:z} as
\begin{align*}
  z &= \sum\limits_{j \in \sigma'} \sum\limits_{x \in \eta_j} \lambda_x x
     + \sum\limits_{i \in \sigma''} \sum\limits_{x \in \eta_i} \lambda_x (x - y(i)) .
\end{align*}
The $\lambda_x$ in the first sum add up to $1$, so we can interpret this first sum as an affine combination, while we think of the second sum as a vector that moves us from this affine combination to the center of the smallest empty circumstack.

\medskip \noindent \emph{Remark.}
Following the terminology in \cite{BaEd17}, we could refer to the vertices with $\lambda_x > 0$ as \emph{front} and the vertices with $\lambda_x \leq 0$ as \emph{back}, noting that the latter can be removed without affecting the radius.

\subsection{Proof of Generalized Discrete Morse Property}
\label{sec:4.3}

The crucial insight that turns the dual solution into a proof that the chromatic radius function is generalized Morse is the following:  when we remove a point $x$ from $\nu$, this only affects Condition \eqref{eq:kkt:lambdaleq} in Lemma~\ref{lem:kkt}.
Therefore, if $\lambda_x \leq 0$, then the smallest empty circumstack of $\nu$ is still the smallest empty circumstack of $\nu \setminus \{x\}$.
The idea is that we remove all points from $\nu$ with non-positive coefficient and thus obtain the minimum of the interval that contains $\nu$.
For this, it is important that we identify the points uniquely, but this is guaranteed by the chromatic genericity of the points, which ascertains that the optimal dual solution is unique.
\begin{lemma}\label{lem: uniqueness_of_dual_solution}
  Let $\chi \colon A \to \sigma$ be a chromatic point set in $\Rspace^d$, and let $z, a, b$ describe an empty stack that passes through the points in $\eta \subseteq A$. 
  If $A$ is chromatically generic, then there exists at most one
  set of parameters $\lambda_x$ and $\mu_j$, with $x \in A$ and $j \in \sigma$, that satisfies the conditions of Lemma~\ref{lem:kkt}.
\end{lemma}
\begin{proof}
  By Conditions \eqref{eq:kkt:lambdaleq} and \eqref{eq:kkt:mugeq}, we have $\lambda_x = 0$ if $x \in A \setminus \eta$ and $\mu_j = 0$ if $j \in \sigma''$, in which we recall that $\sigma = \sigma' \sqcup \sigma''$ and $\sigma'$ are the colors whose spheres have maximum radius.
  It thus suffices to show that the linear relation \eqref{eq:kkt:gradient} restricted to points $x \in \eta$ and colors $j \in \sigma'$,
  \begin{align*}
    \nabla f(z,b) + \sum\nolimits_{x\in \eta} \lambda_x \nabla g_x(z,a) + \sum\nolimits_{j\in\sigma'} \mu_j \nabla h_j(a,b) &= 0 ,
  \end{align*}
  has at most one solution.
  We do this by showing that the $\nabla g_x$ and $\nabla h_j$ are linearly independent.
  Writing these vectors as the columns of a matrix, we perform elementary column operations to make it obvious that the columns are linearly independent.
  First simplify the notation by assuming $\sigma' = \{0,1,\ldots,k\}$, and replace $\nabla h_j$ by $\nabla h_j - \nabla h_0$, for $1 \leq j \leq k$.
  The resulting $k$ vectors are the respective first columns of blocks $1$ to $k$ of the matrix in Table~\ref{tbl:gradients}.
  \begin{table}[t]
    \vspace{-0.1in} {\footnotesize
    \[
      \begin{bNiceArray}{rccc||rcc|c|rcc||ccc|c}[margin,first-col,first-row]
      &\multicolumn{4}{c}{\text{block $0$}}&\multicolumn{3}{c}{\text{block $1$}}& &\multicolumn{3}{c}{\text{block $k$}}&\multicolumn{3}{c}{\text{block $k+1$}} & \\
      1,2,\dots,d& 0&c_0&x-c_0&\ldots& 0 &x-c_0&\ldots&\ldots& 0 &x-c_0&\ldots&c_{k+1}&x-c_{k+1}&\ldots&\ldots\\\hline
           d+1&-1& 1 &     &      & 1 &     &      &\ldots& 1 &     &      &       &         &      &\ldots\\
           d+2&  &   &     &      &-1 &     &      &\ldots&   &     &      &       &         &      &\ldots\\
\vdots\ \ \ \ &  &   &     &      &   &     &      &\vdots&   &     &      &       &         &      &\vdots\\
           d+k&  &   &     &      &   &     &      &\ldots&-1 &     &      &       &         &      &\ldots\\
         d+k+1&  &   &     &      &   &     &      &\ldots&   &     &      &   1   &         &      &\ldots\\
\vdots\ \ \ \ &  &   &     &      &   &     &      &\vdots&   &     &      &       &         &      &\vdots\\\hline
         d+s+2& 1&   &     &      &   &     &      &\ldots&   &     &      &       &         &      &\ldots
        \end{bNiceArray}
    \]
    }
    \vspace{5pt}
    \caption{{\footnotesize The columns in this matrix are the gradient vectors after combining them as explained in the proof of Lemma~\ref{lem: uniqueness_of_dual_solution}.
    Zero entries are left blank.
    There are $s+1$ blocks of columns, one for each color.
    The respective first columns of the first $k+1$ blocks contain $\nabla h_0$ and $\nabla h_j - \nabla h_0$, for $1 \leq j \leq k$.
    The points of colors $0$ to $k$ all lie on the same sphere, and we get vectors by subtracting the same point, $c_0$, from each such point.
    For each color $j \geq k+1$, we get vectors from the points in block $j$ by subtracting an arbitrary but fixed point $c_j \in B_j$.}}
    \label{tbl:gradients}
  \end{table}
  Furthermore, we replace $\nabla g_x$ by $\nabla g_x + \nabla h_j - \nabla h_0$ for every $x \in \eta_j$ and $1 \leq j \leq k$,
  which effectively moves the $1$ in row $d+j+1$ to row $d+1$.
  Recall that the first $d$ coordinates of $\nabla g_x$ are those of $-2x$.
  We may replace them by the coordinates of $x$ without affecting the linear independence of the vectors.
  Finally, choose an arbitrary but fixed $c_0 \in \eta_0$, and replace $\nabla g_x$ by $\nabla g_x - \nabla g_{c_0}$, for all $x \neq c_0$ in $\eta_j$ with $0\leq j\leq k$.
  Similarly, for each other color $j$, with $k+1 \leq j \leq s$, choose an arbitrary but fixed point, $c_j \in \eta_j$, and replace $\nabla g_x$ by $\nabla g_x - \nabla g_{c_j}$, for each $x \in \eta_j \setminus \{c_j\}$; see the first row of the matrix in Table~\ref{tbl:gradients}.

  \smallskip
  It is now easy to see that the columns in the matrix are linearly independent. 
  To begin, collect all columns that start with $x-c_j$.
  Their topmost $d$ positions contain the vectors considered and found linearly independent in Lemma~\ref{lem:equivalent_genericity_conditions} (b).
  All the remaining columns have their unique pivots in the $s+2$ rows below the top $d$ rows, so adding them preserves the linear independence.
\end{proof}

Guaranteeing uniqueness of the dual solution is, indeed, necessary to have the simplices organized in intervals.
See Figure~\ref{fig:non-generic_counterexample} for an example of points that fail to be chromatically generic whose radius function is not generalized discrete Morse. 
The common center of the two circles is the point $z$, which can be expressed as an affine combination satisfying conditions in Lemma~\ref{lem:kkt} in more than one way. 
We need $\lambda_a+\lambda_b = 1$ and $\lambda_c + \lambda_d = 0$. We can express $z$ as a combination of $a,b$, and not use $c,d$ at all, we can start at a combination of $a, b$ to the left of $a$ and then move to $z$ along a positive multiple of the vector $d-c$, or we can start at a combination of $a, b$ to the right of $b$ and move to $z$ along a negative multiple of $d-c$. 
This leads to $(\lambda_a, \lambda_b, \lambda_c, \lambda_d)$ having signatures $(+,+,0,0)$, $(+,-,-,+)$, and $(-,+,+,-)$, respectively.
As explained in the caption of Figure~\ref{fig:non-generic_counterexample}, this ambiguity prevents the formation of intervals.
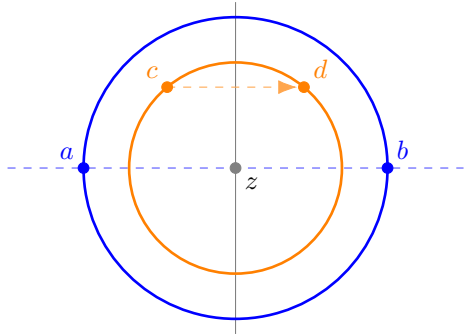
\begin{figure}[hbt]
    \centering
    \begin{tikzpicture}
        \def\radblue{2}
        \def\radred{1.4}
        \coordinate (z) at (0,0);
        \coordinate (p0) at (-1.5*\radblue, 0);
        \coordinate (p1) at (1.5*\radblue, 0);
        \coordinate (a) at (-\radblue, 0);
        \coordinate (b) at (\radblue, 0);
        \coordinate (c) at (90+40:\radred);
        \coordinate (d) at (90-40:\radred);

        \draw[gray] (0, -1.1*\radblue) -- (0, 1.1*\radblue);
        \draw[blue!70, dashed] (p0) -- (p1);
        \draw[orange!70, dashed, -{Latex[length=2.5mm]}, shorten >=2.5pt] (c) -- (d);
        
        \filldraw[gray] (z) circle(2pt) node[below right, black] {$z$};
        \draw[blue, line width=1pt] (z) circle (\radblue);
        \draw[orange, line width=1pt] (z) circle (\radred);

        \filldraw[blue] (a) circle(2pt) node[above left] {$a$};
        \filldraw[blue] (b) circle(2pt) node[above right] {$b$};
        \filldraw[orange] (c) circle(2pt) node[above left] {$c$};
        \filldraw[orange] (d) circle(2pt) node[above right] {$d$};
    \end{tikzpicture}
    \caption{{\footnotesize An example showing that chromatic genericity is needed for the radius function to be generalized discrete Morse.
    Two \emph{blue} points, $a,b$, and two \emph{orange} points, $c,d$, share a common bisector and therefore violate the chromatic genericity condition in Definition~\ref{dfn:chromatic_genericity}.
    Indeed, the four points also lie on a common circle. 
    The two shown circles belong to the smallest empty circumstack of the simplex $abcd$.
    This stack is still the smallest empty circumstack for the edges $ab$, $ad$, and $bc$, but not of any single vertex.
    Hence, the set of simplices that share the radius with $abcd$ does not have a unique minimum and is therefore not an interval.}}
    \label{fig:non-generic_counterexample}
\end{figure}

\smallskip
We are now ready to argue that the radius function on the chromatic Delaunay complex is generalized discrete Morse, provided the points are chromatically generic.
To this end, we construct the interval that contains a given simplex, $\nu$, in the Delaunay complex of $\chi \colon A \to \sigma$.
Starting with the smallest empty circumstack of $\nu$, we add a sphere for every color that is not yet represented and passes through a point at distance at most the radius of the stack from its center.
Specifically, we add the unique sphere that shares the center with the other spheres and passes through the closest point of that color.
Call this the \emph{augmented smallest empty circumstack} of $\nu$, and write $\nu_{\max}$ for the points its spheres pass through.
This augmented stack is the smallest empty circumstack of $\nu_{\max}$.
By Lemma~\ref{lem: uniqueness_of_dual_solution}, there is a unique dual solution, $\lambda, \mu$.
This dual solution assigns a coefficient $\lambda_x$ to each point $x \in A$, and we write $\nu_{\min}$ for the points with $\lambda_x > 0$.
By construction, $\nu \subseteq \nu_{\max}$, and by Condition \eqref{eq:kkt:lambdaleq} in Lemma~\ref{lem:kkt}, $\nu_{\min} \subseteq \nu$.
\begin{theorem}
  \label{thm:chromatic_radius_functions_are_generalized_discrete_Morse}
  Let $A \subseteq \Rspace^d$ be chromatically generic and $\chi \colon A \to \sigma$ a chromatic point set.
  Then the chromatic radius function, $\Radiusf \colon \Delaunay{}{\chi} \to \Rspace$, is generalized discrete Morse.
\end{theorem}
\begin{proof}
  Given $\nu \in \Delaunay{}{\chi}$, we show that the simplices $\nu'$ that satisfy $\nu_{\min} \subseteq \nu' \subseteq \nu_{\max}$ are the unique interval of $\Radiusf$ that contains $\nu$.
  Let $\lambda, \mu$ be the parameters of the dual solution to the smallest empty circumstack of $\nu_{\max}$.
  By Lemma~\ref{lem: uniqueness_of_dual_solution}, $\lambda$ and $\mu$ are unique, and by Lemma~\ref{lem:kkt}, we get the same dual solution of every $\nu'$ that satisfies $\nu_{\min} \subseteq \nu' \subseteq \nu_{\max}$.
  This implies $\Radiusf(\nu') = \Radiusf(\nu)$ for every such $\nu'$.

  \smallskip
  Next we show that $\Radiusf(\nu'') < \Radiusf(\nu)$ for every $\nu'' \subseteq \nu_{\max}$ that does not contain $\nu_{\min}$.
  Consider the smallest empty circumstack of $\nu''$, and let $\lambda'', \mu''$ be the dual solution.
  By Lemma~\ref{lem:kkt}, we get $\lambda''_x \leq 0$ for every $x \in \nu_{\min} \setminus \nu''$, but we have $\lambda_x > 0$ by construction of $\nu_{\min}$.
  By Lemma~\ref{lem: uniqueness_of_dual_solution}, this implies that the smallest empty circumstack of $\nu''$ is different from that of $\nu$.
  Since the two smallest empty circumstacks are different, they have different centers and, by strict convexity, different radii, so $\Radiusf(\nu'') < \Radiusf(\nu)$.

  \smallskip
  It remains to show that $\Radiusf(\nu''') > \Radiusf(\nu)$ for every $\nu''' \supseteq \nu_{\min}$ that is not contained in $\nu_{\max}$.
  The points $y \in \nu''' \setminus \nu_{\max}$ do not lie on the smallest empty circumstack of $\nu_{\max}$, which implies that the smallest empty circumstack of $\nu'''$ is different from that of $\nu$.
  Since $\nu'''$ contains $\nu_{\min}$, the stack of $\nu'''$ must therefore be larger than that of $\nu$.
  As before, the smallest empty circumstacks are different, which implies they have different centers and radii, so $\Radiusf(\nu''') > \Radiusf(\nu)$.
\end{proof}

\noindent \emph{Remark.}
The smallest empty circumstack of a maximum, $\nu_{\max}$, is also the smallest circumstack of $\nu_{\max}$ (without enforcing emptiness).
This is evident from the formulation as an optimization problem:  all constaints $g_x \geq 0$ are inactive so all corresponding $\lambda_x$ vanish.
If we remove these constraints, the same parameters $z, a, b, \lambda, \mu$ still satisfy the KKT conditions and therefore describe the same optimal solution.

\medskip \noindent \emph{Remark.}
Given a chromatically generic $A \subseteq \Rspace^d$, we may compare the Delaunay complexes and their radius functions in the chromatic and the mono-chromatic cases.
As stated in Lemma~\ref{lem:homotopy_equivalence}, the sublevel sets at matching thresholds have the same homotopy type.
Since the type changes whenever we add a critical simplex---which is characterized by $\nu_{\min} = \nu_{\max}$---this implies that the two radius functions have the same critical values.
The optimization perspective reveals that the critical values belong to the same critical simplices. 
Indeed, the smallest empty circumstack of a chromatic critical simplex, $\nu$, is in fact a single circumsphere:  if $\nu_{\max}=\nu_{\min}$, then its dual solution has $\lambda_x > 0$ for all $x\in\nu$, so $\mu_j > 0$ for all $j\in\chi(\nu)$, and all spheres share the same, maximum radius. 
Since the chromatic radius is bounded from above by the mono-chromatic radius, both agree on $\nu$. 
To see that $\nu$ is also critical in the mono-chromatic case, note that problem $\problemP{\nu}$, but with inequalities $h_j\leq 0$ changed to equalities, has the mono-chromatic radius as the optimum. 
The only change in the dual problem is that we lose inequalities on $\mu_j$. Since Lemma~\ref{lem: uniqueness_of_dual_solution} guarantees uniqueness regardless of the inequalities, this implies that $\lambda_x>0$ also for the modified dual problem, so $\nu$ is critical also in the mono-chromatic case.

\section{Persistent Homology of Chromatic Alpha Complexes}
\label{sec:5}

In this section, we review the background needed to turn the chromatic alpha complexes into persistence diagrams, and we advocate the use of six such diagrams, which we refer to as a $6$-pack.
In addition, we discuss the relations between the diagrams in a $6$-pack, as well as relations between different $6$-packs arising from different choices of the subcomplex.

\subsection{Background: Persistent Homology}
\label{sec:5.1}

The goal of this subsection is to introduce the framework of persistent homology \cite{EdHa10}, together with its kernel, image, and cokernel generalizations \cite{CEHM09}. We keep the formalism to a minimum by limiting ourselves to simplicial complexes and $\Zspace / 2\Zspace$ coefficients.

\subsubsection*{Homology with $\mathbb{Z}/2 \mathbb{Z}$ Coefficients}
\label{sec:5.1.1}

Loosely speaking, homology is an algebraic framework that defines and counts holes in a shape. Given a simplicial complex, $K$, a \emph{$p$-chain} is a subset of $p$-simplices. The \emph{sum} of two $p$-chains is the symmetric difference of the two sets: if a $p$-simplex belongs to both chains, the two copies erase each other, as $1+1 = 0$ in modulo-$2$ arithmetic.
The \emph{boundary} of a $p$-simplex is the set of $(p-1)$-dimensional faces, which is a $(p-1)$-chain.
The $p$-chains with the sum operation form a group, $\Cgroup{p}{K}$, and the \emph{boundary operator}, $\partial_p \colon \Cgroup{p}{K} \to \Cgroup{p-1}{K}$, maps a $p$-chain to the sum of its simplices' boundaries.
A~\emph{$p$-cycle} is a $p$-chain with empty boundary, a \emph{filling} of this $p$-cycle is a $(p+1)$-chain whose boundary is the $p$-cycle, and a \emph{$p$-boundary} is a $p$-cycle for which there exits a filling.
The $p$-boundaries and the $p$-cycles form groups by themselves, and since every $p$-boundary is a $p$-cycle, and every $p$-cycle is a $p$-chain, we get three nested groups: $\Bgroup{p}{K} \subseteq \Zgroup{p}{K} \subseteq \Cgroup{p}{K}$.
Two $p$-cycles are \emph{homologous} if their sum has a filling or, equivalently, adding a $p$-boundary to one $p$-cycle gives the other $p$-cycle.
Being homologous is an equivalence relation, whose equivalence classes are the elements of the \emph{$p$-th homology group}: $\Hgroup{p}{K} = \Zgroup{p}{K} / \Bgroup{p}{K}$.
All mentioned groups are vector spaces, so the \emph{ranks} are their dimensions.
Of particular importance is the \emph{$p$-th Betti number} of $K$, which is the rank of the $p$-th homology group: $\rank{\Hgroup{p}{K}} = \rank{\Zgroup{p}{K}} - \rank{\Bgroup{p}{K}}$.

\smallskip
Let $L$ be a subcomplex of $K$. Relative homology describes the connectivity of 
the topological pair $(K,L)$, which geometrically represents $K$ with the subspace $L$ identified as a single point.
The chain groups are the quotients $C_p(K, L) = C_p(K) \big/ C_p(L)$.
Cycle and boundary subgroups are defined as before and, in particular, a relative chain is a relative cycle if its boundary lies in $L$.
Their quotients are the \emph{relative homology groups of the pair}, denoted $\Hgroup{p}{K,L}$.
A convenient algorithm to compute $\Hgroup{p}{K,L}$ reduces the boundary matrix of $K$ from which all rows and columns that correspond to simplices in $L$ are purged. For the spaces considered here, relative homology of a pair is isomorphic to the homology of the quotient space $K\big/L$.
The homology groups and their relative cousins are related by the following long exact sequence:
\begin{align}
  \ldots \to \Hgroup{p}{L} \to \Hgroup{p}{K} \to \Hgroup{p}{K,L} \to \Hgroup{p-1}{L} \to \ldots
    \label{eqn:ESofPair}
\end{align}
A well known property of long exact sequences is that the alternating sum of dimensions of the vector spaces vanishes.
\begin{lemma}
  \label{lem:alternating_sum_of_ranks}
  Let $L \subseteq K$ be simplicial complexes.
  Then
  \begin{align*}
    \sum\nolimits_{p \in \Zspace} (-1)^p [ \rank{\Hgroup{p}{L}} - \rank{\Hgroup{p}{K}} + \rank{\Hgroup{p}{K, L}} ] &= 0.
  \end{align*}
\end{lemma}
\begin{proof}
  By definition of exactness, the rank of each homology group in \eqref{eqn:ESofPair} can be written as the sum of two non-negative integers such that it shares one with the preceding group and the other with the succeeding group along the sequence.
  Since only finitely many groups have non-zero ranks, this implies that the alternating sum of ranks vanishes.
\end{proof}

\subsubsection*{Persistent Homology}
\label{sec:5.1.2}
    
In the following, let $f \colon K \to \Rspace$ be monotonic, with values $r_1<r_2<\ldots<r_n$, and let $K_i = f^{-1} (-\infty, r_i]$ be its $i$-th sublevel set.
Applying the $p$-th homology functor to $\emptyset=K_0\subseteq K_1\subseteq\ldots\subseteq K_n$, we get a sequence of vector spaces:
\begin{align*}
  \Hgroup{p}{K_0} \to \ldots \to \Hgroup{p}{K_{i-1}} \to \Hgroup{p}{K_i} \to \ldots \to \Hgroup{p}{K_{j-1}} \to \Hgroup{p}{K_j} \to \ldots \to \Hgroup{p}{K_n} .
\end{align*}
There is one such sequence for each dimension, $p$.
The inclusions $K_i \subseteq K_j$ induce maps $\fgroup{i}{j} \colon \Hgroup{p}{K_i} \to \Hgroup{p}{K_j}$ for all $0 \leq i \leq j \leq n$.
This sequence is called a \emph{persistence module}.
It can be written as a direct sum of indecomposable modules of the form $\ldots \to 0 \to \kkk \to \ldots \to \kkk \to 0 \to \ldots$, where $\kkk = \Zspace/2\Zspace$, all maps between these $1$-dimensional vector spaces are identities, and all others are zero maps.
Each indecomposable module has a concrete interpretation, namely a \emph{birth} followed by a \emph{death} of a homology class.
Specifically, we have such an indecomposable module from position $i$ to position $j-1$ if
\medskip \begin{itemize}
  \item there is a class, $\gamma \in \Hgroup{p}{K_i}$ that does not belong to the image of $\fgroup{i-1}{i}$, and 
  \item $\fgroup{i}{j-1} (\gamma)$ does not belong to the image of $\fgroup{i-1}{j-1}$, but $\fgroup{i}{j} (\gamma)$ belongs to the image of $\fgroup{i-1}{j}$.
\end{itemize} \medskip
We say $\gamma$ is \emph{born} at $K_i$ and \emph{dies entering} $K_j$.
We record this information with the point $(f(r_i), f(r_j))$, noting that the second coordinate is $\infty$ if the class is born but never dies.
The resulting multi-set of points in the extended plane is the \emph{$p$-th persistence diagram} of $f$, denoted $\Dgm{p}{f}$. Sometimes, we drop the index and write $\Dgm{}{f}$ for the disjoint union of the $\Dgm{p}{f}$ over all dimensions, $p$.
If $L$ is a subcomplex of $K$, we get a filtration, $L_i$, by restricting $f$ to $L$. 
The inclusions of the pairs $(K_0,L_0)\subseteq(K_1,L_1)\subseteq\ldots\subseteq (K_n,L_n)$ give rise to a sequence of relative homology groups,
\begin{align*}
  \Hgroup{p}{K_0,L_0} \to \ldots \to \Hgroup{p}{K_{i-1},L_{i-1}} \to \Hgroup{p}{K_i,L_i} \to \ldots \!\to \Hgroup{p}{K_n,L_n} .
\end{align*}
Applying the above definitions to this sequence yields the \emph{$p$-th relative persistent diagram}.

\smallskip
An important property of the persistence diagram is its stability.
Specifically, the bottleneck distance between the diagrams of $f, g \colon K \to \Rspace$ is bounded from above by the $L_\infty$-distance between the two maps:
\begin{align}
  \Bottleneck{\Dgm{p}{f}}{\Dgm{p}{g}}  &\leq  \| f-g \|_\infty ;
\end{align}
see \cite{CEH07}.
The \emph{persistence} of a point in the persistence diagram is the vertical distance to the diagonal, $| f(r_j) - f(r_i)| $, and the \emph{$1$-norm} of the diagram is the sum of persistences of its points, denoted $\norm{\Dgm{}{f}}_1$.
To cope with points at infinity, we use a \emph{cut-off}, $C$, which we effectively substitute for $\infty$ (and for birth- and death-values larger than the threshold).
This gives finite $1$-norms and preserves relationships implied by exact sequences, as expressed in Theorem~\ref{thm:norm_relations} below.

\subsubsection*{Kernels, Images, and Cokernels}
\label{sec:5.1.3}
   
Let $L \subseteq K$ be simplicial complexes, $\Kfun \colon K \to \Rspace$ monotonic, and $\Lfun \colon L \to \Rspace$ the restriction of $\Kfun$ to $L$.
Taking sublevel sets, we get two parallel persistence modules and maps from one module to the other:
\begin{align*}
  \begin{array}{ccc ccccc ccc}
    \Hgroup{p}{K_0} & \to & \ldots & \to & \Hgroup{p}{K_i} & \to & \Hgroup{p}{K_{i+1}} & \to & \ldots & \to & \Hgroup{p}{K_n} \\
    \uparrow & & \ldots & & \uparrow & & \uparrow & & \ldots & & \uparrow   \\
    \Hgroup{p}{L_0} & \to & \ldots & \to & \Hgroup{p}{L_i} & \to & \Hgroup{p}{L_{i+1}} & \to & \ldots & \to & \Hgroup{p}{L_n} .
  \end{array}
\end{align*}
Write $\kappa_i \colon \Hgroup{p}{L_i} \to \Hgroup{p}{K_i}$ for the vertical maps, which are induced by the inclusions $L_i \subseteq K_i$, for $0 \leq i \leq n$.
These maps have kernels, images, and cokernels, which form persistence modules of their own:
\begin{align*}
  \begin{array}{ccc ccccc ccc}
    \kernel{p}{\kappa_0} & \to & \ldots & \to & \kernel{p}{\kappa_i} & \to & \kernel{p}{\kappa_{i+1}} & \to & \ldots & \to & \kernel{p}{\kappa_n}, \\
    \image{p}{\kappa_0} & \to & \ldots & \to & \image{p}{\kappa_i} & \to & \image{p}{\kappa_{i+1}} & \to & \ldots & \to & \image{p}{\kappa_n}, \\
    \coker{p}{\kappa_0} & \to & \ldots & \to & \coker{p}{\kappa_i} & \to & \coker{p}{\kappa_{i+1}} & \to & \ldots & \to & \coker{p}{\kappa_n} .
  \end{array}
\end{align*}
These persistence modules were introduced and studied in \cite{CEHM09}.
Following the notation in that paper, we write $\Dgm{}{\kernel{}{\Lfun \to \Kfun}}$, $\Dgm{}{\image{}{\Lfun \to \Kfun}}$, and $\Dgm{}{\coker{}{\Lfun \to \Kfun}}$ for the corresponding persistence diagrams.
These diagrams are also stable under perturbations of the monotonic function, and they can be computed efficiently.
We omit details and refer to \cite{CEHM09}, in particular for the matrix reduction algorithms, which we have implemented \cite{DrMa23} to study the meaning of these derived persistence diagrams for chromatic point sets.

\subsection{6-pack of Persistent Diagrams}
\label{sec:5.2}

The main new concept in this section is a collection of six related persistence diagrams, which we use to quantify the way different point sets mingle. We call this collection a \emph{$6$-pack}. A $6$-pack can be defined for any pair of topological spaces $L\subseteq K$ with a filtration on $K$. We explain the construction on a concrete example illustrated in Figure~\ref{fig:circle_on_background}. 
Let $K = \Delaunay{}{\chi}$ be the chromatic Delaunay complex of the portrayed chromatic set, and let $L \subseteq K$ be the blue subcomplex, consisting of those simplices in $K$ that only have blue vertices. 
Let $\Kfun \colon K \to \Rspace$ be the chromatic radius function, and write $\Lfun$ and $\KLfun$ for the restrictions of $\Kfun$ to $L$ and $K \setminus L$.
\begin{figure}[htb]
  \centering \vspace{0.1in}
  \includegraphics[width=.38\textwidth]{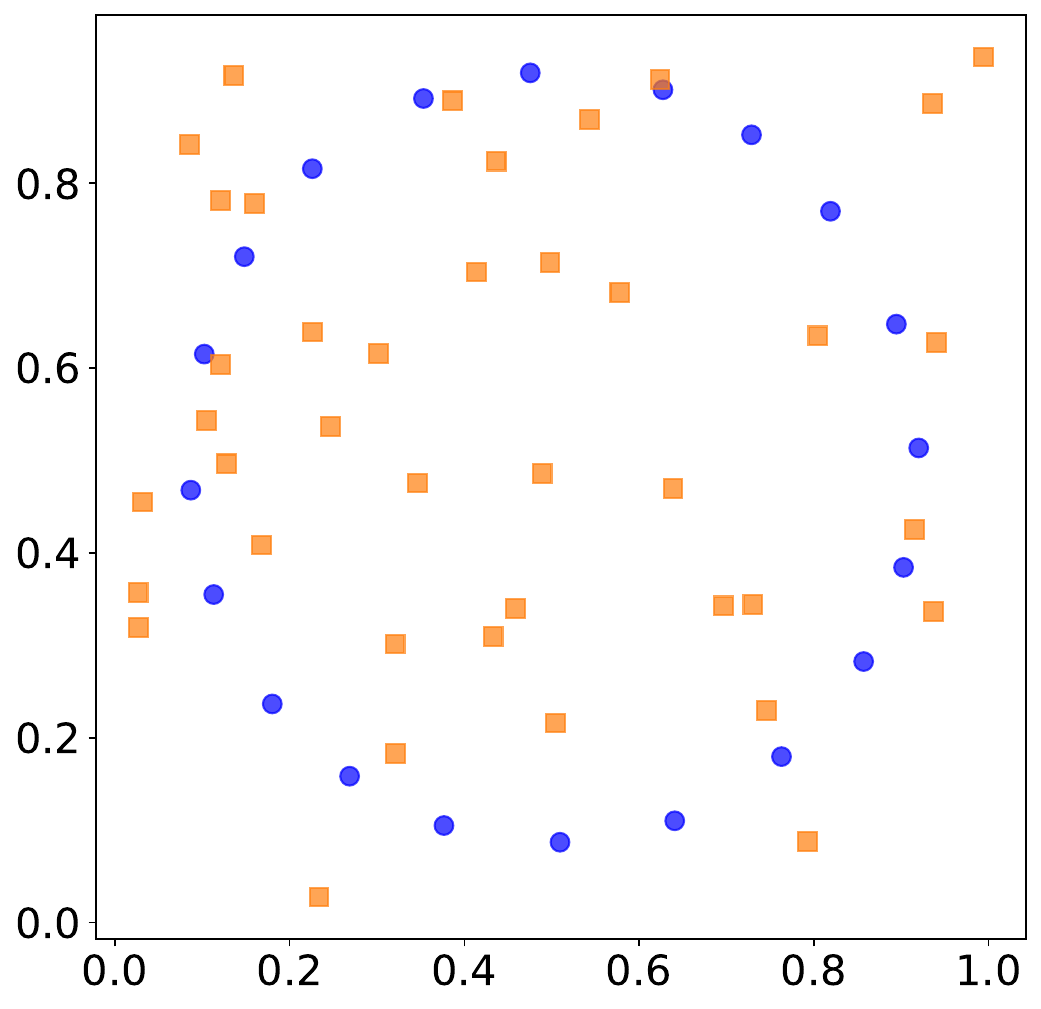}
  \hspace{10mm}
  \includegraphics[width=.38\textwidth]{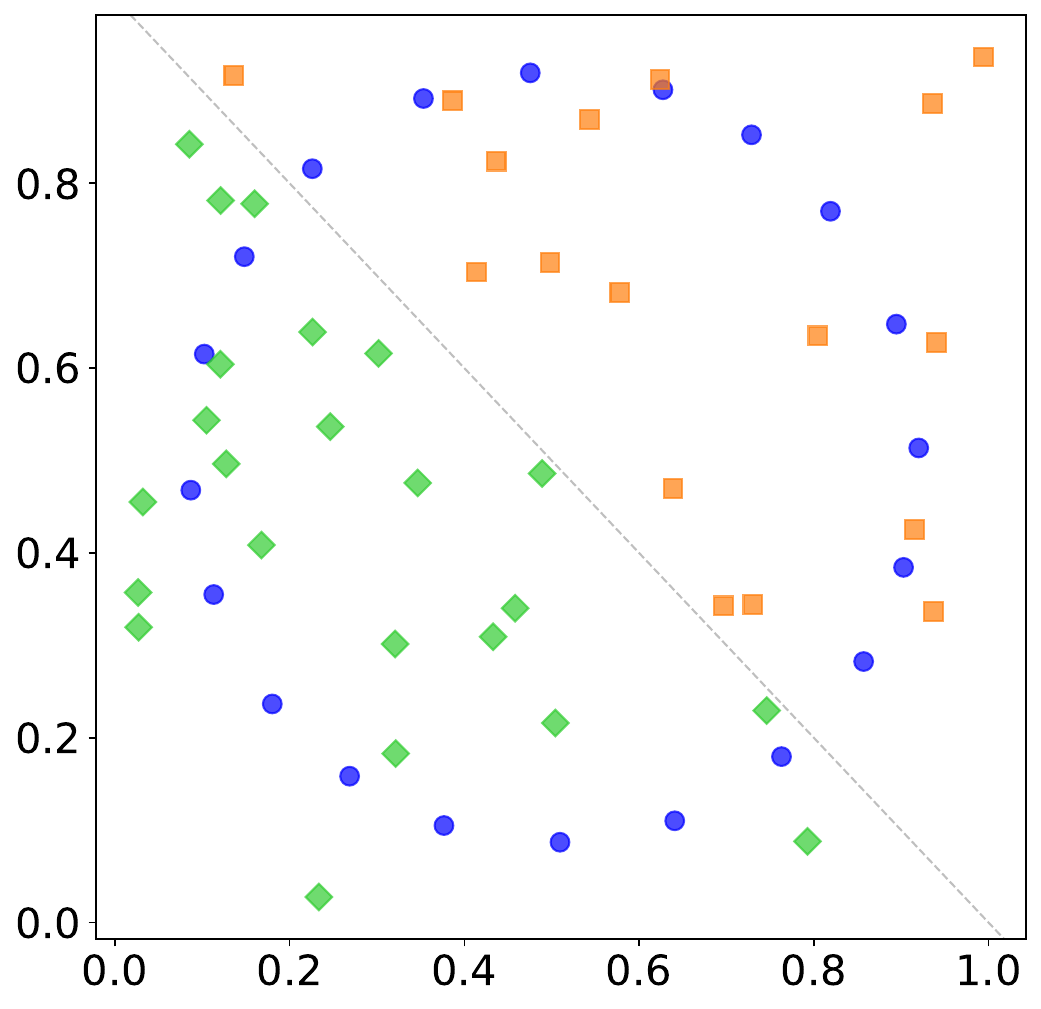}
  \caption{\footnotesize{A bi-chromatic point set on the \emph{left}, and a tri-chromatic point set on the \emph{right}.
  The \emph{dotted} line indicates the separation of \emph{green} from \emph{orange} points that form the background for the \emph{blue} circle.}}
  \label{fig:circle_on_background}
\end{figure}
The radius function and its restrictions give rise to three persistence modules, and we get three additional persistence modules for the kernel, the image, and the cokernel of the map on homology induced by the inclusion $L \subseteq K$; see Section~\ref{sec:5.1.3}.
The persistence diagrams in the $6$-pack are arranged as in Table~\ref{tab:6-pack}, in a manner that lends itself to comparing the information between them.
\begin{table}[htb]
  \centering
  \begin{tabular}{c||c||c}
    \multicolumn{1}{l||}{\emph{kernel:}} & \multicolumn{1}{l||}{\emph{relative:}} & \multicolumn{1}{l}{\emph{cokernel:}} \\
    $\Dgm{}{\kernel{}{\Lfun \to \Kfun}}$ & $\Dgm{}{\KLfun}$ & $\Dgm{}{\coker{}{\Lfun \to \Kfun}}$ \\ \hline \hline
    \multicolumn{1}{l||}{\emph{domain:}} & \multicolumn{1}{l||}{\emph{image:}} & \multicolumn{1}{l}{\emph{codomain:}} \\
    $\Dgm{}{\Lfun}$ & $\Dgm{}{\image{}{\Lfun \to \Kfun}}$ & $\Dgm{}{\Kfun}$ \\ [+1mm]
  \end{tabular}
  \caption{\footnotesize The arrangement of the persistence diagrams in the $6$-pack for the pair $L \subseteq K$ in two rows and three columns.
  Read the six positions in a circle so that the domain lies between the kernel and the image, the image lies between the domain and the codomain, etc.}
  \label{tab:6-pack}
\end{table}

\begin{figure}[htb]
  \centering \vspace{0.1in}
  \includegraphics[width=\textwidth]{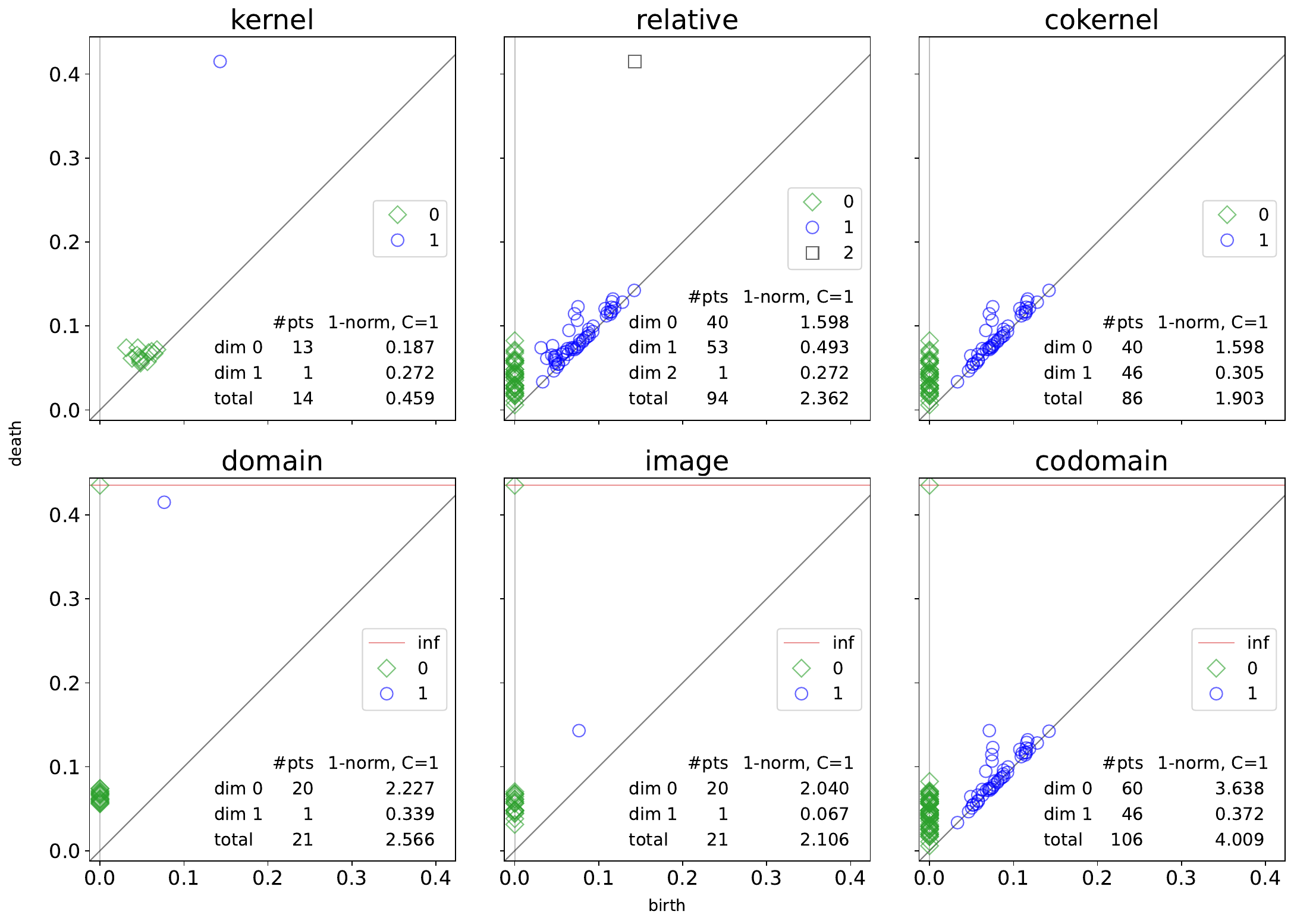}
  \vspace{-0.2in}
  \caption{\footnotesize{The $6$-pack for the bi-chromatic point set in the \emph{left panel} of Figure~\ref{fig:circle_on_background}. 
  The domain, $L$, is the blue subcomplex of the codomain, $K$, which is the $3$-dimensional chromatic Delaunay complex of the \emph{blue} and \emph{orange} points.}}
  \label{fig:6-pack_cb}
\end{figure}

\smallskip
Figure~\ref{fig:6-pack_cb} displays the $6$-pack for the point set in the left panel of Figure~\ref{fig:circle_on_background} with the blue subcomplex chosen as $L$.
Not surprisingly, the circle of blue points gives rise to a persistent $1$-cycle in $L$ captured in the diagram of the domain.
At the time of its birth, this $1$-cycle includes into a non-trivial $1$-cycle in $K$, so we get a point with the same birth-coordinate in the diagram of the image.
When the circle is filled by orange disks, it becomes a trivial $1$-cycle in $K$, which is marked by its death in the image and the simultaneous birth in the kernel.
Eventually, the blue circle is filled by blue disks, so it dies in the domain and simultaneously in the kernel.
To summarize, the point $(a,c)$ in the diagram of the domain splits into two points, $(a,b)$ in the diagram of the image, and $(b,c)$ in the diagram of the kernel.
While the split into two like this is a common phenomenon, not all points split in this manner; see the relations in the next subsection.
The point $(b,c)$ can also be seen one dimension higher in the relative persistence diagram of the pair.
Indeed, there is a non-bounding $2$-cycle in the quotient space once the blue circle is filled by orange disks.
Similarly, the point $(a,b)$ can also be found in the diagram of the codomain.
Both occurrences of $(a,b)$ correspond to the $1$-cycle representing the blue circle in homology, which explains why the point is missing in the diagram of the cokernel.

\smallskip
Note that other natural choices of $L$ are the orange subcomplex or the disjoint union of the blue and orange subcomplexes, which is a choice that is symmetric with respect to the colors. For two colors, these are the three possible $\Gamma$-subcomplexes defined in Section~\ref{sec:3.4}. We now revisit some of these observations in a more general setting, where the pair of topological spaces, $L \subseteq K$, is not necessarily formed by chromatic complexes. 

\subsection{Relations Between Diagrams in a 6-pack}
\label{sec:5.3}

The inclusion of sublevel sets $L_i \subseteq K_i$ induces a map on homology $\kappa_i \colon \Hgroup{}{L_i} \to \Hgroup{}{K_i}$.
This map has a component in each dimension, $p$, and we write $\kernel{p}{\kappa_i}$, $\image{p}{\kappa_i}$, $\coker{p}{\kappa_i}$ for the kernel, image, cokernel of $\kappa_i$ in dimension $p$.
\begin{lemma}
  \label{lem:short_exact_sequences}
  Let $L_i \subseteq K_i$ be simplicial complexes and $\kappa_i \colon \Hgroup{}{L_i} \to \Hgroup{}{K_i}$ the induced map on homology.
  For each dimension, $p$, there are short exact sequences
  \begin{align}
    0 &\to \kernel{p}{\kappa_i} \to \Hgroup{p}{L_i} \to \image{p}{\kappa_i} \to 0, 
    \label{eqn:ranksg} \\
    0 &\to \image{p}{\kappa_i} \to \Hgroup{p}{K_i} \to \coker{p}{\kappa_i} \to 0, 
    \label{eqn:ranksf} \\
    0 &\to \coker{p}{\kappa_i} \to \Hgroup{p}{K_i,L_i} \to \kernel{p-1}{\kappa_i} \to 0.
    \label{eqn:ranksm}
  \end{align}
\end{lemma}
\begin{proof}
  The first two exact sequences are obvious from the definitions and the isomorphism theorem.
  To see the third exact sequence, we recall the long exact sequence of the pair; see Equation (\ref{eqn:ESofPair}).
  Working with field coefficients, all homology groups are vector spaces and thus split.
  In particular, $\Hgroup{p}{L_i} \simeq \kernel{p}{\kappa_i} \oplus \image{p}{\kappa_i}$, in which $\kernel{p}{\kappa_i}$ and $\image{p}{\kappa_i}$ are the images of the incoming and outgoing maps.
  We can therefore substitute $\kernel{p}{\kappa_i} \to 0 \to \image{p}{\kappa_i}$ for $\Hgroup{p}{L_i}$.
  By the same token, we substitute $\image{p}{\kappa_i} \to 0 \to \coker{p}{\kappa_i}$ for $\Hgroup{p}{K_i}$, and we remove $0 \to \image{p}{\kappa_i} \to \image{p}{\kappa_i}$ to get
  \begin{align*}
    \ldots \to \kernel{p}{\kappa_i} \to 0 &\to \coker{p}{\kappa_i} \to \Hgroup{p}{K_i,L_i} \to \kernel{p-1}{\kappa_i} \to 0 \to \coker{p-1}{\kappa_i} \to \ldots,
  \end{align*}
  which contains the required third short exact sequence.
\end{proof}
It follows that the ranks of $\kernel{p}{\kappa_i}$ and $\image{p}{\kappa_i}$ add up to the rank of $\Hgroup{p}{L_i}$, etc.
This implies relations between the $1$-norms of corresponding persistence diagrams.
\begin{theorem}
  \label{thm:norm_relations}
  Let $L \subseteq K$ be simplicial complexes, $\Kfun \colon K \to \Rspace$ monotonic, and $\Lfun$, $\KLfun$ the restrictions of $\Kfun$ to $L$ and $K \setminus L$.
  For each dimension, $p$, and any fixed cut-off for the $1$-norms, $C>0$,
  \begin{align}
    \norm{\Dgm{p}{\Lfun}}_1  &=  \norm{\Dgm{p}{\kernel{}{\Lfun \to \Kfun}}}_1 + \norm{\Dgm{p}{\image{}{\Lfun \to \Kfun}}}_1 ,
      \label{eqn:onenormg} \\
    \norm{\Dgm{p}{\Kfun}}_1  &=  \norm{\Dgm{p}{\image{}{\Lfun \to \Kfun}}}_1 + \norm{\Dgm{p}{\coker{}{\Lfun \to \Kfun}}}_1 , 
      \label{eqn:onenormf} \\
    \norm{\Dgm{p}{\KLfun}}_1  &=  \norm{\Dgm{p}{\coker{}{\Lfun \to \Kfun}}}_1 + \norm{\Dgm{p-1}{\kernel{}{\Lfun \to \Kfun}}}_1 .
      \label{eqn:onenormm}
  \end{align}
\end{theorem}
\begin{proof}
  We prove \eqref{eqn:onenormg}.
  We write $0 \leq r_1,r_2,\ldots,r_n$ for the values of $\Kfun$ smaller than $C$.
  In addition, set $r_0 = - \infty$ and use the cut-off $r_{n+1}=C$ for the $1$-norms. 
  Letting $L_i = \Lfun^{-1} [0, r_i]$, note that $L_i = \Lfun^{-1} [0,r]$ for all $r_i \leq r < r_{i+1}$, and hence the ranks of the various groups are constant between two contiguous values.
  We can therefore write the $1$-norm of $\Dgm{p}{\Lfun}$ as a sum of $n$ contributions, and similar for the $1$-norms of the kernel and image diagrams:
  \begin{align}
    \norm{\Dgm{p}{\Lfun}}_1  &= \sum\nolimits_{i=0}^{n} (r_{i+1}-r_i) \, \rank{\Hgroup{p}{L_i}} , \\
    \norm{\Dgm{p}{\kernel{}{\Lfun \to \Kfun}}}_1  &= \sum\nolimits_{i=0}^{n} (r_{i+1}-r_i) \, \rank{\kernel{p}{\kappa_i}} , \\
    \norm{\Dgm{p}{\image{}{\Lfun \to \Kfun}}}_1  &= \sum\nolimits_{i=0}^{n} (r_{i+1}-r_i) \, \rank{\image{p}{\kappa_i}} .
  \end{align}
  We thus get \eqref{eqn:onenormg} from \eqref{eqn:ranksg}.
  With the same argument applied to $K_i$, image, and cokernel, we get \eqref{eqn:onenormf} from \eqref{eqn:ranksf}, and applied to $(K_i,L_i)$, cokernel, and kernel, we get \eqref{eqn:onenormm} from \eqref{eqn:ranksm}.
\end{proof}

We note that similar equations do not hold for the $0$-norm, which counts the points in the diagrams.
Putting the equations in Theorem~\ref{thm:norm_relations} together yields a vanishing alternating sum:
\begin{align*}
  \sum\nolimits_{p \in \Zspace} (-1)^p \left[ \norm{\Dgm{p}{\Lfun}}_1  - \norm{\Dgm{p}{\Kfun}}_1 + \norm{\Dgm{p}{\KLfun}}_1 \right]  &=  0 .
\end{align*}
While there are relations between the diagrams in a 6-pack, no single diagram is necessarily determined by the others. Figure~\ref{fig:counterexample} shows one such example.
\begin{figure}[htb]
  \centering \vspace{0.1in}
  \resizebox{!}{2.0in}{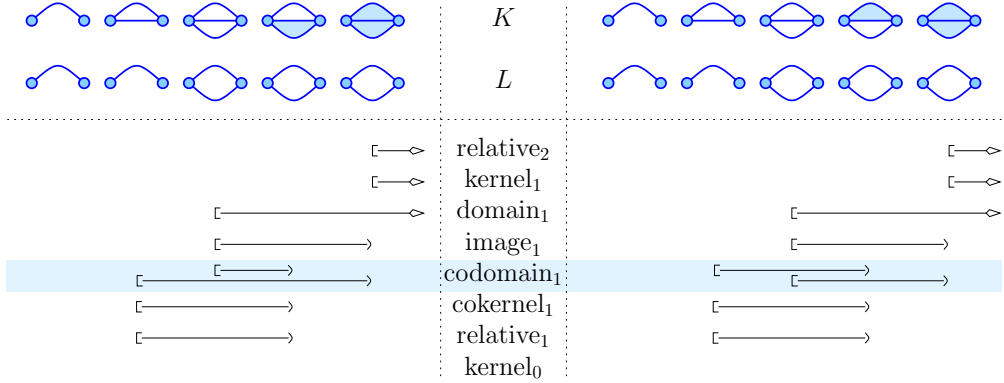}
  \caption{\footnotesize{Example showing that five diagrams do not imply the sixth.
    The two filtrations differ by a single $2$-dimensional cell added in the respective fourth steps of the filtrations.
    Correspondingly, five of the $1$-dimensional persistence diagrams (shown as barcodes) are the same, while the highlighted diagrams of the codomain differ on the two sides.}}
  \label{fig:counterexample}
\end{figure}

Further relations among the diagrams in a $6$-pack are suggested by the case-by-case analysis for the simultaneous occurrence of births and deaths in various groups provided in \cite{CEHM09}. For example, consider the triple $\kernel{}{\kappa_i}$, $\Hgroup{}{L_i}$, $\image{}{\kappa_i}$. At a given radius, the rank of each group can change by at most one. The short exact sequence (\ref{eqn:ranksg}) reduces the twenty-six non-trivial combinations of changes down to only six. Out of those, \cite{CEHM09} gives examples for five of them and shows that the sixth, death-nothing-birth, cannot occur because a death in $\kernel{}{\kappa_i}$ always implies a death in $\Hgroup{}{L_i}$. This is an additional relation, which is not implied directly by (\ref{eqn:ranksg}). The same case is excluded for the triple in (\ref{eqn:ranksf}). Analogously, one can show that the case death-nothing-birth is excluded for the triple $\coker{p}{\kappa_i}$, $\Hgroup{p}{K_i,L_i}$, $\kernel{p-1}{\kappa_i}$.

\subsection{Relations Between 6-packs of a Triplet}
\label{sec:5.4}

The framework described so far is amenable to a pair of complexes $L \subseteq K$, filtered by a monotonic function. This section addresses the next simplest case: when we have a sequence of three nested complexes, $M \subseteq L \subseteq K$, which gives rise to four long exact sequences:
\begin{align}
  \ldots \to \Hgroup{p}{L} \to \Hgroup{p}{K} &\to \Hgroup{p}{K,L} \to \Hgroup{p-1}{L} \to \ldots ,\label{eq:les_LK} \\
  \ldots \to \Hgroup{p}{M} \to \Hgroup{p}{K} &\to \Hgroup{p}{K,M} \to \Hgroup{p-1}{M} \to \ldots ,\label{eq:les_MK} \\
  \ldots \to \Hgroup{p}{M} \to \Hgroup{p}{L} &\to \Hgroup{p}{L,M} \to \Hgroup{p-1}{M} \to \ldots ,\label{eq:les_ML} \\
  \ldots \to \Hgroup{p}{L,M} \to \Hgroup{p}{K,M} &\to \Hgroup{p}{K,L} \to \Hgroup{p-1}{L,M} \to \ldots\label{eq:les_rel} .
\end{align}
\begin{figure}[htb]
  \centering \vspace{0.1in}
  \resizebox{!}{1.4in}{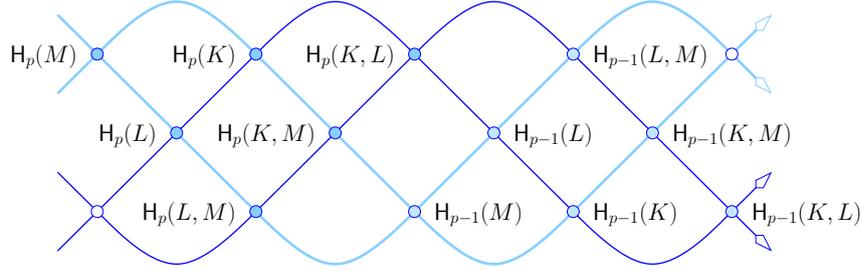}
  \caption{\footnotesize{The four exact sequences for three complexes drawn along sine-like curves in the plane.
  After each half-period, the dimension of the homology group drops by one.}}
  \label{fig:sine}
\end{figure}
To shed light on how they relate to each other, we draw them as sine-like curves, each directed from left to right, with the homology groups sitting at the crossings between the curves; see Figure~\ref{fig:sine}.
Observe that the upper left triangular diagram commutes, which implies
\begin{align}
  \kernel{}{[\Hgroup{p}{M} \to \Hgroup{p}{L}]} &\subseteq \kernel{}{[\Hgroup{p}{M} \to \Hgroup{p}{K}]} , 
    \label{eqn:3kernel} \\
  \image{}{[\Hgroup{p}{M} \to \Hgroup{p}{K}]} &\subseteq \image{}{[\Hgroup{p}{L} \to \Hgroup{p}{K}]} 
    \label{eqn:3image}
\end{align}
for all dimensions $p$.
Similar inclusions follow from the commutativity of the other regions in the arrangement of curves. 
The four inclusions that give rise to the sequences (\ref{eq:les_LK}) to (\ref{eq:les_rel}) yield four $6$-packs, among which six diagrams appear twice, namely, $\Dgm{}{f_K}$, $\Dgm{}{f_L}$, $\Dgm{}{f_M}$, $\Dgm{}{f_{K,L}}$, $\Dgm{}{f_{K,M}}$, $\Dgm{}{f_{L,M}}$. 
Therefore, we have eighteen unique diagrams, some of which are closely related.

\subsection{A Tri-chromatic Case Study}
\label{sec:5.5}
    
While two colors give rise to interesting patterns, more colors do more so.
With the increase in the number of colors, there is an explosive increase of configurations to study.
We suggest looking at the relations between \emph{$k$-chromatic subcomplexes} of $\Delaunay{}{\chi}$, which are the subcomplexes composed of all simplices with at most $k$ colors, as defined in Section~\ref{sec:3.4}.
In this section, we focus on the tri-chromatic case, with colors $\sigma = \{0,1,2\}$. 
Let $M$ be the mono-chromatic subcomplex, $L$ the bi-chromatic subcomplex, and $K$ the full tri-chromatic Delaunay complex.
As before, $\Kfun \colon K \to \Rspace$ is the chromatic squared radius function, and $\Lfun$, $\Mfun$, $\KLfun$, $\KMfun$, $\LMfun$ are its restrictions.
A cycle can be formed by points of $1$, $2$, or $3$ colors, and it can be filled by points of $0$, $1$, or $2$ additional colors.
Requiring that the sum of two numbers is at most $3$, we get the six mingling types sketched in Figure~\ref{fig:patterns-filled}.
Note that these six patterns are not independent.
For example, the pattern 1+2 also gives rise to pattern 1+0, because the cycle gets filled by its own color eventually.
However, different patterns corresponding to the same cycle will generally have different persistence, which quantifies which patterns is a better fit.
Without a claim on completeness, we list where in the $6$-packs one can find prominent cases of each of these six patterns.
\begin{description}
  \item[{\sc Case} 1+0] $\Dgm{}{\Mfun}$. The complex $M$ is the disjoint union of the three mono-chromatic Delaunay complexes. The diagram records the mono-chromatic cycles.
  \item[{\sc Case} 2+0] $\Dgm{}{\coker{}{\Mfun \to \Lfun}}$. The complex $L$ contains all mono- and bi-chromatic cycles, and it shares the former with $M$. Therefore, we look at the cokernel to keep only the cycles that need two colors to be formed. 
  A cycle dies either when it is filled by its own two colors, or when one of the two colors suffices to form a homologous cycle.
  \item[{\sc Case} 3+0] $\Dgm{}{\coker{}{\Lfun \to \Kfun}}$. As in the previous case, we look at the cokernel to capture cycles that are formed by all three colors, but not yet by any two.
  \item[{\sc Case} 1+1] $\Dgm{}{\kernel{}{\Mfun \to \Lfun}}$. When a cycle formed by one color is filled by another color, a birth in this diagram occurs, and the feature persists until it is filled by its own color.
  \item[{\sc Case} 2+1] $\Dgm{}{\kernel{}{\LMfun \to \KMfun}}$. The idea is similar to Case 1+1: we look at cycles formed by two colors that are filled when also using the third. 
  Unlike in the previous case, we consider the quotient spaces to filter out the mono-chromatic cycles.
  \item[{\sc Case} 1+2] $\Dgm{}{\coker{}{\LMfun \to \KMfun}}$. Mono-chromatic $p$-cycles filled by the other colors appear in the pair $(K,M)$ as $(p+1)$-cycles.
  Those that are filled by exactly one other color also appear in $(L,M)$. We use the cokernel to filter them out.
\end{description}

\begin{figure}[htb]
  \centering \vspace{0.1in}
  \includegraphics[width=\textwidth]{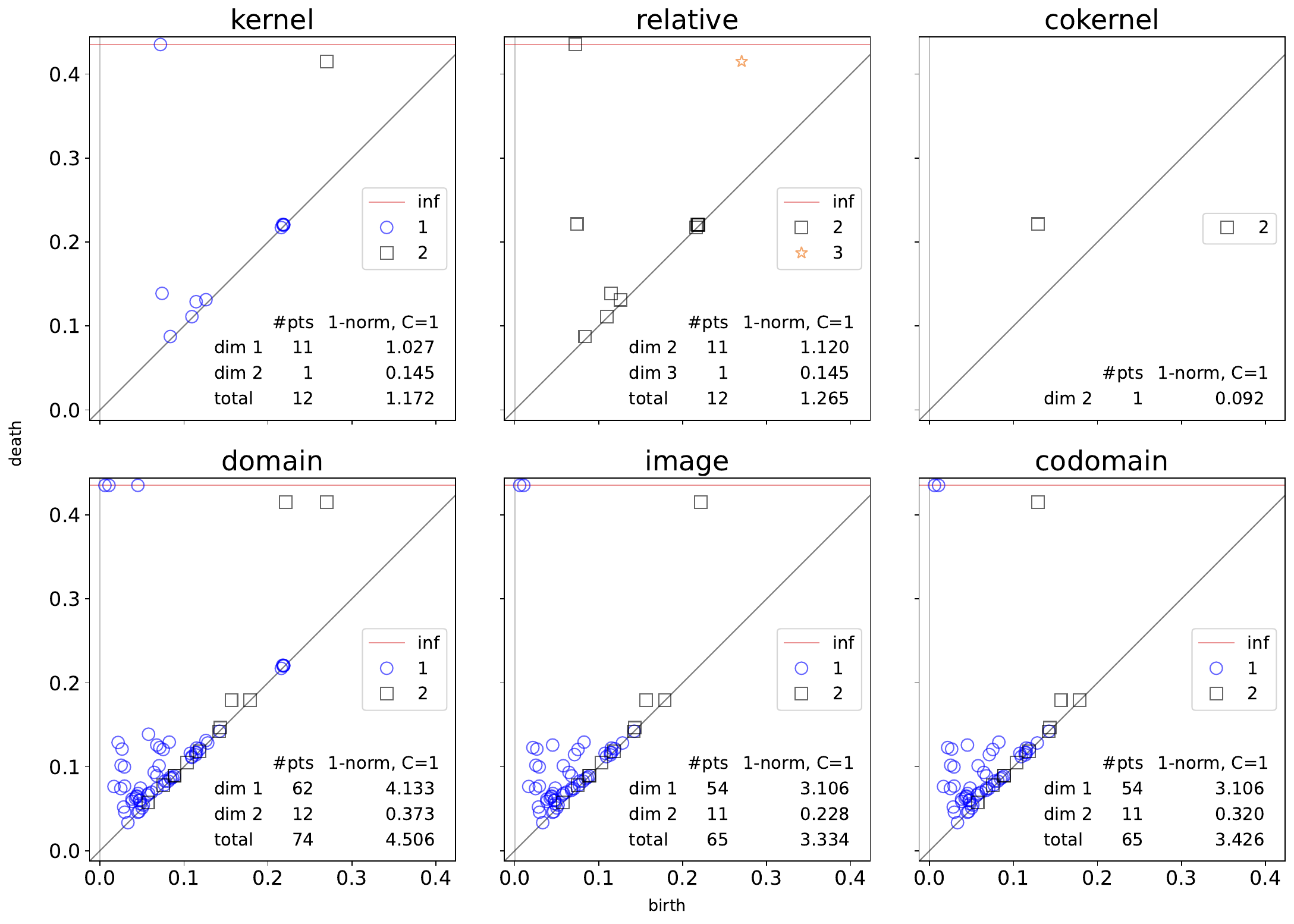}
  \vspace{-0.2in}
  \caption{\footnotesize{The $6$-pack of $(L,M) \subseteq (K,M)$ for the data in the \emph{right panel} of Figure~\ref{fig:circle_on_background}.
  $M$, $L$, and $K$ are the $1$-, $2$- and $3$-chromatic subcomplexes of the chromatic Delaunay complex.}}
  \label{fig:6-pack_csb}
\end{figure}

\smallskip
We now look more closely at the concrete example displayed in the right panel of Figure~\ref{fig:circle_on_background}: a circle of blue points with split background of green and orange points; compare with the mingling pattern 1+2.
Focusing on this pattern, we search for the $6$-pack of the inclusion of the pairs $(L,M) \subseteq (K,M)$ in Figure~\ref{fig:6-pack_csb}.
As suggested in Case~1+2 above, we expect a clear signal in the cokernel diagram, and indeed we see a single prominent point representing a $2$-dimensional relative class.
By construction, this class is born when the mono-chromatic $1$-cycle is filled with two extra colors, and its persistence indicates how much longer it takes to fill the $1$-cycle with just one extra color.
Compare this with the even more prominent point in the diagram of the codomain, $\Dgm{}{f_{K,M}}$.
This point represents the same $1$-cycle, but it expresses different information because it is not sensitive to whether the $1$-cycle is filled by one or two additional colors.

\smallskip
It is interesting to interpret the two high persistence points in the diagram of the domain, which records classes in $\Hgroup{}{L_i,M_i}$.
Since the background consists of two colors, it fills the blue $1$-cycle with only one additional color twice, once with green and another time with orange.
Both classes die at the same moment, namely when the blue cycle is filled by its own color.

\smallskip
As described in the caption of Figure~\ref{fig:patterns-filled}, Case~1+2 could also be interpreted as two overlapping instances of Case~1+1. For example, consider mixing the orange and green points in the right panel of Figure~\ref{fig:circle_on_background}. This pattern is captured in $\Dgm{}{\Lfun}$ as explained in Figure~\ref{fig:bi-chromatic_subcomplex_2d_3_colors}.

\section{Discussion}
\label{sec:6}

The main contribution of this paper is the extension of the theory of alpha complexes to the setting where points are assigned a label. We prove structural results about the radius function on the chromatic Delaunay complex and provide an implementation that facilitates its use in applications. The work reported in this paper suggests new directions of mathematical research aimed at solidifying our understanding of the chromatic setting. 
We list three possible directions.
\smallskip \begin{itemize}
  \item Develop a chromatic variant of Forman's \textbf{discrete Morse theory} \cite{For98}.
  Two concrete questions are the extension of the collapsibility of the \v{C}ech complex to the alpha complex proved in the mono-chromatic case \cite{BaEd17} and the further collapse of $\Alpha{r}{\chi}$ to $\Alpha{r}{A}$.
  \item In many biological questions, the mingling between different populations of cells changes over time; see e.g.\ the study of cell segregation in early development \cite{Mai12} and an early topological approach in \cite{KeEd13}.
  It would therefore be useful to extend the \textbf{vineyard algorithm} \cite{CEM06} to the chromatic setting introduced in this paper.
  \item Applications in material science suggest to relax the focus on the nearest point and base the theory on order-$k$ Voronoi tessellations \cite{Fej76,ShHo75}, for $k$ possibly larger than $1$.
  Among the different options, we favor the construction in which the \textbf{order-$k$ chromatic Delaunay complex} is isomorphic to the nerve of the order-$k$ Voronoi tessellation of the chromatically lifted points.
\end{itemize} \smallskip
In applications to atomic structures, it is furthermore useful to model with different size balls.
This can be done by basing the chromatic construction on the extension of Delaunay complexes to points with real weights, known under a variety of names, such as \emph{regular triangulations} \cite{GKZ94} or \emph{weighted Delaunay complexes} \cite[Section~III.3]{EdHa10}.
Letting $w \colon A \to \Rspace$ assign the weights, the \emph{weighted square distance} of a point $x \in \Rspace^d$ to $a \in A$ is $\Edist{x}{a}^2 - w(a)$.
In molecular biology, $w(a)$ would typically be the square root of the van der Waals radius so that the zero set of the power distance is the sphere with this radius.
With this notion, all geometric structures can be defined as before.
In particular, the \emph{weighted chromatic Delaunay complex} is isomorphic to the weighted Delaunay complex for the chromatically lifted points, and the sublevel set of the (weighted) radius function for $r$ serves as the discrete representation of the union of balls in which the ball centered at $a \in A$ has squared radius $w(a) + r^2$.
We drop further details while mentioning that the difference between weighted and unweighted theories is small, and minor changes suffice to adapt the software for constructing the $6$-packs to the weighted case.


\clearpage
\appendix

\section{Homotopy Equivalence of Pairs}
\label{apx:he_pairs}

We use the notation of \cite{BKRR23}, particularly of Section~3, and describe an additional consequence of the proof of Theorem~B/3.11 in a particular case relevant to our work. Recall that $\mathsf{ClConv}_\ast$ is a category with the following data. Objects are $(X, \mathcal{C}_\ast)$, where $X$ is a subset of  $\Rspace^d$, and $\mathcal{C}_\ast=(\mathcal{C}, (c_\nu)_{\nu\in\Nerve{\mathcal{C}}})$ is its finite cover by closed convex sets, $\mathcal{C}=(C_i)_{i\in I}$, together with a fixed point $c_\nu\in \bigcap_{i \in \nu} C_i$ for each $\nu\subseteq I$ such that the intersection is non-empty. The morphisms $(f, \varphi): (X, \mathcal{C}_\ast) \rightarrow (Y, \mathcal{D}_\ast)$ comprise of an affine linear map $f: X\rightarrow Y$ and a map of cover indices $\varphi: I\rightarrow J$ such that $f(C_i)\subseteq D_{\varphi(i)}$ and $f(c_\nu) = d_{\varphi(\nu)}$ for each $\nu\subseteq I$.
\begin{proposition}\label{prop:homotopy_of_pairs}
    Let $(X, \mathcal{C}_\ast) \overset{(\iota,\varphi)}{\rightarrow} (Y, \mathcal{D}_\ast)$ be a morphism in $\mathsf{ClConv}_\ast$ such that:
    \begin{enumerate}
        \item\label{prop:homotopy_of_pairs:1} $\varphi:[n]\longrightarrow [m]$ is the identity inclusion, where $[n]$, $[m]$ are index sets for the covers $\mathcal{C}$, $\mathcal{D}$,
        \item\label{prop:homotopy_of_pairs:2} $\iota: X \rightarrow Y$ is an inclusion,
        \item\label{prop:homotopy_of_pairs:3} $X \cap D_i = C_i$ for all $1\leq i\leq n$, and $X \cap D_j = \emptyset$ for all $n < j \leq m$.
    \end{enumerate}
    Then $\Nerve{\mathcal{C}} \subseteq \Nerve{\mathcal{D}}$ and the maps $\gamma$, $\psi$ defined in \cite[Section~3]{BKRR23} for $(Y, \mathcal{D}_\ast)$ give a homotopy equivalence between pairs $(Y, X)$ and $(\Nerve{\mathcal{D}}, \Nerve{\mathcal{C}})$. That is, if $H_t, K_t$ are the homotopies between $\gamma\circ\psi$ and $\mathrm{id}_Y$, and $\psi\circ\gamma$ and $\mathrm{id}_{|\Nerve{\mathcal{D}}|}$, respectively, then $H_t(X)\subseteq X$ and $K_t(|\Nerve{\mathcal{C}}|)\subseteq |\Nerve{\mathcal{C}}|$ for all $t\in [0,1]$.
\end{proposition}

Before the proof, note that the unions of the convex variants of Voronoi planks, $\VCPlank{r}{a}{\chi}{\Gamma}$, defined in Section~\ref{sec:3.4} of this paper satisfy the above assumptions both for the inclusion going from $\Gamma$ to a supercomplex $\Delta$, and for the one going from radius $r$ to a larger radius $R$. Proposition~\ref{prop:homotopy_of_pairs} then implies the following
\begin{corollary}
  \label{cor:quotients_homotopy_equivalence}
  For $\Gamma \subseteq \Delta$ and $r \leq R$, there are homotopy equivalences satisfying the following commutative diagram
\[
    \begin{tikzcd}
    \bigcup\nolimits_{a \in A} \CPlank{R}{a}{\chi}{\Delta} / \bigcup\nolimits_{a \in A} \CPlank{R}{a}{\chi}{\Gamma}
    \arrow[r, leftrightarrow, "\simeq"] &
    \AlphaSub{R}{\chi}{\Delta} / \AlphaSub{R}{\chi}{\Gamma} \\
    \bigcup\nolimits_{a \in A} \CPlank{r}{a}{\chi}{\Delta} / \bigcup\nolimits_{a \in A} \CPlank{r}{a}{\chi}{\Gamma}
    \arrow[u, hook] \arrow[r, leftrightarrow, "\simeq"] &
    \AlphaSub{r}{\chi}{\Delta} / \AlphaSub{r}{\chi}{\Gamma} \arrow[u, hook]
    \end{tikzcd}
\]
\end{corollary}

\begin{proof}[Proof of Proposition~\ref{prop:homotopy_of_pairs}]
    We follow \cite[Section~3]{BKRR23} using the same notation and terminology. 
    We describe why the homotopies $H_t, K_t$ are invariant on the subspaces.

    \smallskip
    The construction of $\psi$ requires an extension of the closed cover $\mathcal{D}$ to an open cover $\mathcal{V}$ with the same nerve; see Lemma~3.2 in \cite{BKRR23}.
    Assumption~(\ref{prop:homotopy_of_pairs:3}) of the proposition allows us to construct the extension so that $\mathcal{U}=(U_i = V_i\cap X)$ is the analogous open extension of $\mathcal{C}$. 
    Therefore, $\psi \colon Y \to|\Nerve{\mathcal{D}}|$ restricted to $X$ is the homotopy $\psi_X \colon X \to |\Nerve{\mathcal{C}}|$.
    The map $\gamma \colon |\Nerve{\mathcal{D}}| \to Y$ restricted to $X$ is the homotopy $\gamma_X \colon |\Nerve{\mathcal{C}}| \to X$---this is the naturality of the map proved in Theorem~3.11 \cite{BKRR23}.

    \smallskip
    The homotopy $H_t$ between $\gamma \circ \psi$ and $\mathrm{id}_Y$ is described in the proof of Theorem~3.1 \cite{BKRR23} as a straight line homotopy.
    If $x\in X$, from the above we have $\gamma \circ \psi(x) \in X$. 
    There is some $i$ such that $x\in D_i$, and since $\gamma\circ\psi$ is carried by identity, $\gamma\circ\psi(x)\in D_i$. 
    Since $C_i = D_i\cap X$ by (\ref{prop:homotopy_of_pairs:3}), we have $\gamma \circ \psi(x) \in C_i$. 
    As also $\mathrm{id}_Y(x)\in C_i$, the whole line connecting the two points is in $C_i \subseteq X$. 
    This shows that $H_t(X)\subseteq X$ for all~$t$.
    The homotopy $K_t$ between $\psi \circ \gamma$ and $\mathrm{id}_{|\Nerve{\mathcal{D}}|}$ is constructed in Proposition~3.8 \cite{BKRR23} inductively, following paths defined by contractions of intersections of closed stars. As argued in the proof of Lemma~3.4 \cite{BKRR23}, those intersections are star-shaped with respect to the barycenter of the defining vertices of the stars.
    This implies that the contractions defined for the case of $\Nerve{\mathcal{C}}$ are restrictions of the contractions defined for the case of $\Nerve{\mathcal{D}}$, and, therefore, the final homotopy $K_t$ satisfies $K_t(|\Nerve{\mathcal{C}}|)\subseteq |\Nerve{\mathcal{C}}|$ for all $t$.
\end{proof}

\end{document}